\documentclass[preprint,12pt]{elsarticle}

\pdfoutput=1

\usepackage[english]{babel}
\usepackage{graphics}
\usepackage{graphicx}
\usepackage{epsfig}
\usepackage{amssymb}
\usepackage{amsmath}
\usepackage{amsthm}
\usepackage{amsfonts}
\usepackage{mathrsfs}

 \newtheorem{corollary}{Corollary}

\newtheorem{theorem}{Theorem}
\newtheorem{proposition}{Proposition}
\newtheorem{example}{Example}
\newtheorem{remark}{Remark}

\begin{document}

\begin{frontmatter}

\title{Decay properties for functions of matrices over $C^*$-algebras}

\author[emory]{Michele Benzi\fnref{grant}}
\ead{benzi@mathcs.emory.edu}

\author[xlim]{Paola Boito}
\ead{paola.boito@unilim.fr}

\fntext[grant]{Work supported in part by NSF Grant DMS 1115692.}

\address[emory]{Department of Mathematics and Computer Science,
Emory University, Atlanta, GA 30322, USA. }
\address[xlim]{Equipe Calcul Formel, DMI-XLIM UMR 7252 Universit\'e de Limoges - CNRS,
123 avenue Albert Thomas, 87060 Limoges Cedex, France.}

\begin{abstract}
We extend previous results on the exponential off-diagonal decay
of the entries of analytic functions of banded and sparse
matrices to the case where the matrix entries are elements
of a $C^*$-algebra.
\end{abstract}
  
  
\end{frontmatter}

  

\begin{abstract}
We extend previous results on the exponential off-diagonal decay
of the entries of analytic functions of banded and sparse
matrices to the case where the matrix entries are elements
of a $C^*$-algebra.
\end{abstract}



\section{Introduction}
Decay properties of inverses, exponentials and other
functions of band or sparse
matrices over $\mathbb{R}$ or $\mathbb{C}$ have been investigated
by several authors in recent years
\cite{BG99,BR07,DMS84,highambook,iserles,jaffard,KSW,mastronardi}.
Such properties play an important role in various
applications including electronic structure
computations in quantum chemistry \cite{BBR13,BM12},
quantum information theory \cite{CE06,CEPD06,ECP10,SCW06},
high-dimensional statistics \cite{Aune}, random matrix
theory \cite{Molinari}
and numerical analysis \cite{BC13,ye13}, to name a few.

Answering a question posed by P.-L.~Giscard and coworkers \cite{giscard},
we consider generalizations of existing decay estimates
to functions of matrices with entries in more general
algebraic structures than the familiar fields $\mathbb{R}$
or $\mathbb{C}$. In particular, we propose extensions
to functions of matrices with entries from
the following algebras:

\begin{enumerate}
\item Commutative algebras of complex-valued continuous functions;
\item Non-commutative algebras of bounded operators on
a complex Hilbert space;
\item The real division algebra $\mathbb{H}$ of quaternions.
\end{enumerate}

The theory of complex $C^*$-algebras provides the natural abstract setting
for the desired extensions \cite{KRI,MC,Rudin}. 
Matrices over such algebras arise naturally in various application
areas, including parametrized linear systems and eigenproblems
\cite{CGI10,Tobler}, differential equations \cite{CT06},
generalized moment problems \cite{Osipov},
control theory \cite{Bunce,Curtain1,Curtain2}, and quantum
physics \cite{Berezansky71,Berezansky98,Doran}. The study of matrices
over $C^*$-algebras is also of independent mathematical interest;
see, e.g., \cite{Grove,Kadison}.

Using the holomorphic functional
calculus, we establish exponential
off-diagonal
decay results for analytic functions of banded
$n\times n$
Hermitian matrices over $C^*$-algebras, both commutative and
non-commutative. Our decay estimates are expressed
in the form of computable bounds on the norms of the entries
of $f(A)$ where $A=[a_{ij}]$ is an $n\times n$ matrix with entries
$a_{ij} = a_{ji}^*$ in
a $C^*$-algebra $\mathcal{A}_0$ and $f$ is an analytic function defined on a
suitable open subset of $\mathbb{C}$
 containing the spectrum of $A$, viewed as an element of the
$C^*$-algebra $\mathcal{M}_n(\mathcal{A}_0)$ ($=\mathcal{A}_0^{n\times n}$). 
The interesting case
is when the constants in the bounds do not depend on $n$.
Functions of more general sparse matrices over $\mathcal{A}_0$
will also be discussed.

For the case of functions of $n\times n$ quaternion matrices, we identify 
the set of such matrices with a 
(real) subalgebra of $\mathbb{C}^{2n\times 2n}$ and treat them as a
special type of complex block matrices; as we will see, this will
impose some restrictions on the type of functions that we are
allowed to consider.

\section{Background on $C^*$-algebras}

In this section we provide definitions and notations used throughout
the remainder of the paper, and recall some of the fundamental results
from the theory of $C^*$-algebras. The reader is referred to \cite{MC,Rudin}
for concise introductions to this theory and to \cite{KRI} for a 
more systematic treatment.

Recall that a {\em Banach algebra} is a complex algebra $\mathcal{A}_0$
with a norm making $\mathcal{A}_0$ into a Banach space and satisfying
$$\|ab\| \le \|a\|\|b\|$$
for all $a,b\in \mathcal{A}_0$.
In this paper we consider only {\em unital} Banach algebras, i.e., 
algebras with a multiplicative unit $I$ with $\|I\|=1$.

An {\em involution} on a Banach algebra $\mathcal{A}_0$ is a map
$a\mapsto a^*$ of $\mathcal{A}_0$ into itself satisfying
\begin{itemize}
\item [(i)] $(a^*)^* = a$
\item [(ii)] $(ab)^* = b^*a^*$
\item [(iii)] $(\lambda a + b)^* = \overline{\lambda} a^* + b^*$
\end{itemize}
for all $a,b\in \mathcal{A}_0$ and $\lambda \in \mathbb{C}$.
A {\em $C^*$-algebra} is a Banach algebra with an involution such that
the $C^*$-{\em identity}
$$ \|a^*a\| = \|a\|^2$$
holds for all $a\in \mathcal{A}_0$.
Note that we do not make any assumption on whether $\mathcal{A}_0$
is commutative or not.

Basic examples of $C^*$-algebras are:

\begin{enumerate}
\item The commutative algebra $C(\mathcal{X})$ of all continuous complex-valued functions
on a compact Hausdorff space $\mathcal{X}$. Here the addition and multiplication 
operations are defined pointwise, and the norm is given by 
$\|f\|_{\infty} = \max_{t\in \mathcal{X}} |f(t)|$.
The involution on $C(\mathcal{X})$ maps each function $f$ to its complex conjugate $f^*$,
defined by $f^*(t) = \overline{f(t)}$ for all $t\in \mathcal{X}$.
\item The algebra $\mathcal{B}(\mathcal{H})$ of all bounded linear
operators on a complex Hilbert space $\mathcal{H}$, with the operator
norm $\|T\|_{op} = \sup \|Tx\|_{\mathcal{H}} /\|x\|_{\mathcal{H}}$, 
where the supremum is taken over all nonzero $x\in {\mathcal{H}}$.
The involution on $\mathcal{B}(\mathcal{H})$ maps each bounded linear
operator $T$ on $\mathcal{H}$ to its adjoint, $T^*$.
\end{enumerate}

Note that the second example contains as a special case the 
algebra $\mathcal{M}_n (\mathbb{C})$ ($= \mathbb{C}^{k\times k}$)
of all $k\times k$ matrices with
complex entries, with the norm being the usual spectral norm
and the involution mapping each matrix $A=[a_{ij}]$ to its
Hermitian conjugate $A^* = [\, \overline{a_{ji}}\, ]$. This algebra
is noncommutative for $k\ge 2$.

Examples 1 and 2 above provide, in a precise sense, 
the \lq\lq only\rq\rq\ examples of $C^*$-algebras. Indeed,
every (unital) commutative $C^*$-algebra admits a faithful representation
onto an algebra of the form $C(\mathcal{X})$ for a suitable (and essentially
unique) compact Hausdorff space $\mathcal{X}$; and, similarly, every unital
(possibly noncommutative) $C^*$-algebra can be faithfully represented as
a norm-closed subalgebra of $\mathcal{B}(\mathcal{H})$ for a suitable
complex Hilbert space $\mathcal{H}$.

More precisely, a map $\phi$ between two
$C^*$-algebras is a $*$-{\em homomorphism} if $\phi$ is linear,
multiplicative, and such that $\phi(a^*) = \phi(a)^*$; a
$*$-{\em isomorphism} is a bijective $*$-homomorphism. 
Two $C^*$-algebras are said to be {\em isometrically
$*$-isomorphic} if there is a norm-preserving $*$-isomorphism
between them, in which case they are indistinguishable as
$C^*$-algebras.  A {\em $*$-subalgebra} $\mathcal{B}_0$ of a 
$C^*$-algebra is a subalgebra that is {\rm $*$-closed}, i.e.,
$a\in \mathcal{B}_0$
implies $a^*\in \mathcal{B}_0$. Finally, a $C^*$-subalgebra is a
norm-closed $*$-subalgebra of a $C^*$-algebra.   
The following two results are classical \cite{G,GN}. 

\begin{theorem}\label{thm_gelfand} \textnormal{(Gelfand)}
Let $\mathcal{A}_0$ be a commutative $C^*$-algebra. Then there is
a compact Hausdorff space $\mathcal{X}$ such that $\mathcal{A}_0$ is 
isometrically $*$-isomorphic to $C(\mathcal{X})$. If $\mathcal{Y}$ is another
compact Hausdorff space such that $\mathcal{A}_0$ is
isometrically $*$-isomorphic to $C(\mathcal{Y})$, then $\mathcal{X}$ 
and $\mathcal{Y}$
are necessarily homeomorphic.
\end{theorem}

\begin{theorem}\label{thm_gn} \textnormal{(Gelfand--Naimark)}
Let $\mathcal{A}_0$ be a $C^*$-algebra. Then there is a
complex Hilbert space $\mathcal{H}$ such that $\mathcal{A}_0$
is isometrically $*$-isomorphic to a $C^*$-subalgebra
of $\mathcal{B}(\mathcal{H})$.
\end{theorem}

We will also need the following definitions and facts.
An element $a\in \mathcal{A}_0$ of a $C^*$-algebra
is {\em unitary} if $aa^*=a^*a = I$, 
{\em Hermitian}
(or {\em self-adjoint}) if $a^* = a$,
{\em skew-Hermitian} if $a^* = -a$, {\em normal} if $aa^* = a^*a$. 
Clearly, unitary, Hermitian and skew-Hermitian elements are 
all normal. Any element $a\in \mathcal{A}_0$
can be written uniquely as $a = h_1 + {\rm i}\,h_2$ with $h_1,h_2$
Hermitian and $\rm i = \sqrt{-1}$. 

For any (complex) Banach algebra
$\mathcal{A}_0$, the {\em spectrum}
of an element $a \in \mathcal{A}_0$ is the set of all
$\lambda\in \mathbb{C}$ such that $\lambda I - a$ is
not invertible in $\mathcal{A}_0$. We denote the
spectrum of $a$ by $\sigma(a)$. For any $a\in \mathcal{A}_0$,
the spectrum $\sigma(a)$ is a non-empty compact subset
of $\mathbb{C}$ contained in the closed disk of radius
$r=\|a\|$ centered at $0$. 
 The complement 
$r(a) = \mathbb{C} \backslash \sigma(a)$ of the spectrum
of an element $a$ of a $C^*$-algebra is called
the {\em resolvent set} of $a$.
The {\em spectral radius}
of $a$ is defined as $\rho(a) = \max\{|\lambda| \,;\, \lambda\in \sigma(A)\}$. 
{\em Gelfand's formula} for the spectral radius \cite{G} states that
\begin{equation}\label{GF}
\rho(a) = \lim_{m\to \infty} \|a^m\|^{\frac{1}{m}}\,.
\end{equation}
Note that this identity contains the statement that the
above limit exists.

If $a \in \mathcal{A}_0$
(a $C^*$-algebra) is Hermitian, $\sigma(a)$ is
a subset of $\mathbb{R}$. 
If $a \in \mathcal{A}_0$ is normal (in particular,
Hermitian), then $\rho(a) = \|a\|$. This implies that
if $a$ is Hermitian, then either $-\|a\|\in \sigma(a)$
or $\|a\|\in \sigma(a)$. 
The spectrum of a skew-Hermitian element is purely imaginary,
and the spectrum of a unitary element 
is contained in the unit circle $S^1 = \{z\in \mathbb{C}\, ; \,
 |z|=1\}$.


An element $a\in \mathcal{A}_0$
is {\em nonnegative} if $a=a^*$ and the spectrum of
$a$ is contained in $\mathbb{R}_+$, the nonnegative 
real axis. Any linear combination with real nonnegative
coefficients of nonnegative elements of a $C^*$-algebra
is nonnegative; in other words, the set of all nonnegative
elements in a $C^*$-algebra $\mathcal{A}_0$ form a 
(nonnegative) {\em cone} in $\mathcal{A}_0$. For any 
$a\in \mathcal{A}_0$, $aa^*$ is
nonnegative, and $I + aa^*$ is invertible in $\mathcal{A}_0$. 
Furthermore, $\|a\| = \sqrt{\rho(a^*a)} = 
\sqrt{\rho(aa^*)}$, for any $a\in \mathcal{A_0}$.

Finally, we note that if $\|\cdot\|_*$ and $\|\cdot\|_{**}$
are two norms with respect to which $\mathcal{A}_0$ is a
$C^*$-algebra, then $\|\cdot\|_* = \|\cdot\|_{**}$. 

\section{Matrices over a $C^*$-algebra}
Let $\mathcal{A}_0$ be a $C^*$-algebra. 
Given a positive integer $n$, let $\mathcal{A} = \mathcal{M}_n(\mathcal{A}_0)$ 
be the set 
of $n\times n$ matrices with entries in $\mathcal{A}_0$. 
Observe that $\mathcal{A}$ has a natural $C^*$-algebra structure, 
with matrix addition and multiplication 
defined in the usual way 
(in terms, of course, of the corresponding operations on $\mathcal{A}_0$). 
The involution is naturally defined as follows:
given a matrix $A=[a_{ij}] \in \mathcal{A}$, the adjoint of $A$
is given by $A^* = [a_{ji}^*]$. The algebra $\mathcal{A}$ is
obviously unital, with unit
$$I_n = \left[ \begin{array}{cccc}
               I & 0 &\ldots  & 0 \\ 0 & \ddots & \ddots & \vdots \\
               \vdots  & \ddots & \ddots & 0 \\ 0 & \ldots & 0 & I
             \end{array}
      \right] 
$$
where $I$ is the unit of $\mathcal{A}_0$.
The definition of unitary, Hermitian,
skew-Hermitian and normal matrix are the obvious ones.

It follows from the Gelfand--Naimark representation theorem 
(Theorem \ref{thm_gn} above) that
each $A\in \mathcal{A}$ can be represented as a matrix $T_A$ of bounded
linear operators, where $T_A$ acts on the direct sum 
$\mathscr{H} = \mathcal{H} \oplus \cdots \oplus \mathcal{H}$ of $n$ copies of 
a suitable complex Hilbert space $\mathcal{H}$. 
This fact allows us to introduce
an operator norm on $\mathcal{A}$, defined as follows:
\begin{equation}\label{op_norm1}
\|A\|:= \sup_{\|x\|_{\mathscr{H}}=1} \|T_Ax\|_{\mathscr{H}}\,,
\end{equation}
where
$$\|x\|_{\mathscr{H}}:=\sqrt{\|x_1\|_{\mathcal{H}}^2 + 
                             \cdots +
                             \|x_n\|_{\mathcal{H}}^2}$$
is the norm of an element $x=(x_1,\ldots ,x_n)\in \mathscr{H}$.
Relative to this norm, $A\in \mathcal{A}$ is a $C^*$-algebra.
Note that $\mathcal{A}$ can also be identified with the tensor
product of $C^*$-algebras $\mathcal{A}_0\, \otimes \,  \mathcal{M}_n (\mathbb{C})$.

Similarly, Gelfand's theorem (Theorem \ref{thm_gelfand} above) implies that
if $\mathcal{A}_0$ is commutative, there is a compact Hausdorff space
$\mathcal{X}$ such that any $A\in \mathcal{A}$ can be identified with
a continuous matrix-valued function
$$A: \mathcal{X} \longrightarrow \mathcal{M}_n(\mathbb{C})\,.$$
In other words, $A$ can be represented as an $n\times n$ matrix of 
continuous, complex-valued
functions: $A = [a_{ij}(t)]$, with domain $\mathcal{X}$. 
The natural $C^*$-algebra
norm on $\mathcal{A}$, which can be identified with
$ C(\mathcal{X}) \otimes \mathcal{M}_n (\mathbb{C})$, is now the operator norm
\begin{equation}\label{op_norm2}
\|A\|:= \sup_{\|x\|=1} \|Ax\|\,,
\end{equation}
where $x = (x_1,\ldots ,x_n)\in [C(\mathcal{X})]^n$ has norm 
$\|x\|=\sqrt{\|x_1\|_{\infty}^2 + \cdots + \|x_n\|_{\infty}^2}$
with $\|x_i\|_{\infty}= \max_{t\in \mathcal{X}}|x_i(t)|$, for $1\le i\le n$.

Since $\mathcal{A}$ is a $C^*$-algebra, all the definitions and basic facts about the
spectrum remain valid for any matrix $A$ with entries in $\mathcal{A}_0$.
Thus, the spectrum $\sigma(A)$ of $A\in \mathcal{A}$ is the set of 
all $\lambda\in \mathbb{C}$
such that $\lambda I_n - A$ is not invertible in $\mathcal{A}$. 
If $0\in \sigma(A)$, we will also say that $A$ is {\em singular}. The
set $\sigma(A)$ is a nonempty compact subset of $\mathbb{C}$ completely
contained in the disk of radius $\|A\|$ centered at $0$. The definition
of spectral radius and Gelfand's formula (\ref{GF}) remain valid. 
Hermitian matrices
have real spectra, skew-Hermitian matrices have purely imaginary
spectra, unitary matrices have spectra contained in $S^1$, and so forth. 
Note, however, that it is not true in general that a normal matrix $A$ over a
$C^*$-algebra can be unitarily diagonalized \cite{Kadison}.

In general, it is difficult to estimate the spectrum of a matrix
$A=[a_{ij}]$ over a $C^*$-algebra. It will be useful for what follows
to introduce the {\em matricial norm} of $A$ \cite{ostrowski,robert}, 
which is defined as
the $n\times n$ real nonnegative matrix 
$$\hat A = \left[ \begin{array}{cccc}
               \|a_{11}\| & \|a_{12}\| &\ldots  & \|a_{1n}\| \\ 
               \|a_{21}\| & \|a_{22}\| & \ldots & \|a_{nn}\| \\
               \vdots  & \ddots & \ddots & \vdots \\ 
               \|a_{n1}\| & \|a_{n2}\| &\ldots  & \|a_{nn}\| 
             \end{array}
      \right] \,.
$$

The following result shows that we can obtain upper bounds on the
spectral radius and operator norm of a matrix $A$ over a $C^*$-algebra
in terms of the more easily computed corresponding quantities  
for $\hat A$. As usual, the symbol $\|\cdot\|_2$ denotes the
spectral norm of a matrix with real or complex entries.

\begin{theorem}\label{thm_matricial_norm}
For any $A\in \mathcal{A}$, the following inequalities hold:
\begin{enumerate}
\item $\|A\| \le \|\hat A\|_2$;
\item $\rho (A) \le \rho (\hat A)$.
\end{enumerate}
\end{theorem}
\begin{proof}
To prove the first item, observe that 
$$\|A\| = \sup \|T_Ax\|_{\mathscr{H}}
= \sup \left [ \sum_{i=1}^n \left  \|\sum_{j=1}^n \pi (a_{ij}) x_j\right \|^2 
\right ]^\frac{1}{2}\,,
$$
where $\pi(a_{ij})$ is the Gelfand--Naimark representation of $a_{ij}\in \mathcal{A}_0$
as a bounded operator on $\cal H$, and
the sup is taken over all $n$-tuples $(x_1, x_2, \ldots ,x_n)\in
{\mathscr{H}}$ with $\|x_1\|_\mathcal{H}^2 + \|x_2\|_\mathcal{H}^2 + \cdots
+ \|x_n\|_\mathcal{H}^2  = 1$. 

Using the triangle inequality and the fact that the Gelfand--Naimark map is
an isometry we get
$$\|A\| \le \sup \left [ \sum_{i=1}^n \left (\sum_{j=1}^n \|a_{ij}\|
\|x_j\|_\mathcal{H}\right )^2 \right ]^\frac{1}{2}\,,$$
or, equivalently,
$$\|A\|\le \sup_{(\xi_1,\ldots ,\xi_n)\in \Xi^n} \left [ \sum_{i=1}^n \left (\sum_{j=1}^n \|a_{ij}\|
\xi_j \right )^2 \right ]^\frac{1}{2}\,,
$$
where $\Xi^n:= \{(\xi_1, \ldots , \xi_n) \,|\, \xi_i \in \mathbb{R}_{+} \,\, \forall i
\,\,\, \textnormal{and} \,\,\, \sum_{i=1}^n \xi_i^2 = 1 \}$.
On the other hand,
$$\|\hat A\|_2 = 
\sup_{(\xi_1,\ldots ,\xi_n)\in S^n} \left [ \sum_{i=1}^n \left |\sum_{j=1}^n \|a_{ij}\|
\xi_j \right |^2 \right ]^\frac{1}{2}\,,
$$
where $S^n$ denotes the unit sphere in $\mathbb{C}^n$. Observing that
$\Xi^n \subset S^n$, we conclude that $\|A\| \le \|\hat A\|_2$.

To prove the second item we use the characterizations 
$\rho (A) = \lim_{m\to \infty}\|A^m\|^{\frac{1}{m}}$, 
$\rho (\hat{A}) = \lim_{m\to \infty}\|\hat{A}^m\|_2^{\frac{1}{m}}$ 
and the fact that $\|A^m\|\le \|\widehat{A^m}\|_2$, which we just
proved. A simple
inductive argument shows that 
$\|\widehat{A^m}\|_2 \le \|\hat{A}^m\|_2$ for all $m = 1, 2, \ldots $,
thus yielding the desired result.
\end{proof}

\begin{remark}
A version of item 1 of the previous theorem was 
proved by A.~Ostrowski in \cite{ostrowski}, in the context of matrices
of linear operators on normed vector spaces. Related results can also
be found in \cite{Gil} and \cite{Tretter}.
\end{remark}

\begin{remark}\label{remark2}
If $A$ is Hermitian or (shifted) skew-Hermitian, then $\hat A$ is 
real symmetric. In this case
$\|A\|=\rho(A)$, $\|\hat A\|_2 = \rho (\hat A)$ and item 2 reduces
to item 1. On the other hand, in the more general case where $A$ is normal,
the matrix $\hat A$ is not necessarily symmetric or even normal and we 
obtain the bound
$$\|A\| = \rho(A)\le \rho(\hat A)\,,$$
which is generally better than $\|A\|\le \|\hat A\|_2$.
\end{remark}

Next, we prove a simple invertibility condition for 
matrices over the commutative $C^*$-algebra $C(\mathcal{X})$.

\begin{theorem}\label{thm_invert}
A matrix $A$ over $\mathcal{A}_0 = C(\mathcal{X})$ is invertible in
$\mathcal{A}=\mathcal{M}_n(\mathcal{A}_0)$
if and only if  for each $t\in \mathcal{X}$
the $n\times n$ matrix $A(t) = [a_{ij}(t)]$ is
invertible in $\mathcal{M}_n(\mathbb{C})$.
\end{theorem}
\begin{proof}
The theorem will be proved if we show that 
\begin{equation}\label{spectra}
\sigma (A) = {\bigcup}_{t\in \mathcal{X}}\sigma (A(t))\,.
\end{equation}
Assume first that $\lambda\in \sigma(A(t_0))$ for some $t_0\in \mathcal{X}$.
Then $\lambda I_n - A(t_0)$ is not invertible, therefore 
$\lambda I_n - A(t)$ is not invertible for all $t\in \mathcal{X}$, and
$\lambda I_n - A$ fails to be invertible as an element of
$\mathcal{A}=\mathcal{M}_n(\mathcal{A}_0)$. This shows that
$${\bigcup}_{t\in \mathcal{X}}\sigma (A(t))\subseteq \sigma (A).$$
To prove the reverse inclusion, we show that the resolvent sets
satisfy 
$${\bigcap}_{t\in \mathcal{X}}r(A(t)) \subseteq r(A).$$
Indeed, if $z\in r(A(t))$ for all $t\in \mathcal{X}$, then
the matrix-valued function $t\mapsto (zI_n - A(t))^{-1}$
is well defined and necessarily continuous on $\mathcal{X}$.
Hence, $z\in r(A)$. This completes the proof.
\end{proof}

Clearly, the set $\mathcal{K} = \{A(t)\,;\, t\in \mathcal{X}\}$ is compact 
in $\mathcal{M}_n (\mathbb{C})$. The spectral radius, 
as a function on $\mathcal{M}_n (\mathbb{C})$, is continuous and
thus the function $t\mapsto \rho(A(t))$ is continuous on the
compact set $\mathcal{X}$. It follows that this function attains its
maximum value for some $t _0\in \mathcal{X}$.
The equality (\ref{spectra}) above then implies that
$$\rho(A) = \rho(A(t_0)) = \max_{t\in \mathcal{X}} \rho(A(t))\,.$$
If $A$ is normal, we obviously have 
\begin{equation}\label{norms}
\|A\| = \max_{t\in \mathcal{X}} \|A(t)\|_2\,.
\end{equation}
Recalling that for any $A\in \mathcal{A}$ the norm satisfies 
$\|A\| = \sqrt{\rho(AA^*)} = \sqrt{\rho(A^*A)}$, we conclude
that the identity (\ref{norms}) holds for any matrix $A$ over
the $C^*$-algebra $C(\mathcal{X})$.

As a special case, consider an $n\times n$ matrix with entries
in $C(\mathcal{X})$, where $\mathcal{X}=[0,1]$. Each entry $a_{ij} = a_{ij}(t)$
of $A$ is a continuous complex-valued function of $t$. 
One can think of such an $A$ in
different ways. As a mapping of $[0,1]$ into $\mathcal{M}_n(\mathbb{C})$,
$A$ can be regarded as a continuous curve in the space of all
$n\times n$ complex matrices. On the other hand, $A$ is also a
point in the $C^*$-algebra of matrices over $C(\mathcal{X})$. 

Theorem \ref{thm_invert} then states that $A$ is invertible if and only
if the corresponding curve $\mathcal{K} = \{A(t)\,;\, t\in [0,1]\}$ does not
intersect the set of singular $n\times n$ complex matrices, i.e., if and only if
$\mathcal{K}$ is entirely contained in the group $\mathcal{GL}_n(\mathbb{C})$
of invertible $n\times n$ complex matrices.


\begin{example}
As a simple illustration of Theorem \ref{thm_matricial_norm}, 
consider the $2\times 2$ Hermitian matrix 
over $C([0,1])$:
$$ A = A(t) = \left [ \begin{array}{cc} 
               {\rm e}^{-t} & t^2 + 1 \\
               t^2 + 1 & {\rm e}^t
\end{array}
\right ]\,.
$$
Observing now that 
$$\hat A =  \left [ \begin{array}{cc}
               1 & 2  \\
               2 & {\rm e}
\end{array}
\right ]\,,
$$
we obtain the bound $\rho(A) = \|A\|\le \|\hat A\|_2 \approx 4.03586$.

By direct computation we find that
${\rm det}(\lambda I_2 - A(t)) = \lambda^2 -2 (\cosh t) \lambda - t^2(t^2+ 2)$
and thus the spectrum of $A(t)$
(as an element of the $C^*$-algebra of matrices over $C([0,1])$)
consists of all the numbers of the form
$$ \lambda_{\pm} (t)= \cosh t \pm \sqrt{\cosh^2 t + t^2(t^2 + 2)}, 
\quad 0\le t \le 1$$
(a compact subset of $\mathbb{R}$). 
Also note that ${\rm det}(A(t)) = - t^2(t^2 + 2)$ vanishes for $t=0$, showing that
$A$ is {\em not} invertible in $\mathcal{M}_n(C([0,1])$.

Finding the maxima and minima over $[0,1]$ of the continuous functions 
$\lambda_{-}(t)$
and $\lambda_{+}(t)$ we easily find that the spectrum of $A(t)$
is given by
$$\sigma(A(t)) = [-0.77664, 0]\cup [2, 3.86280],$$
where the results have been rounded to five decimal digits. 
Thus, in this simple example $\|\hat A\|_2=4.03586$ gives a pretty good
upper bound for the true value $\|A\| = 3.86280$.  
\end{example}

\section{The holomorphic functional calculus}
The standard way to define the notion of an analytic function $f(a)$ 
of an element $a$ of a $C^*$-algebra ${\cal A}_0$ is via contour integration. 
In particular, we can use this approach to define functions of a
matrix $A$ with elements in ${\cal A}_0$.

Let $f(z)$ be a complex function which is analytic 
in a open neighborhood $U$ of $\sigma(a)$. Since $\sigma(a)$ is compact, 
we can always find  a finite collection $\Gamma=\cup_{j=1}^\ell\gamma_j$ 
of smooth simple closed curves  
whose interior parts contain $\sigma(a)$ and
entirely contained in $U$. The curves 
$\gamma_j$ are assumed to be oriented counterclockwise.

Then 
$f(a)\in\mathcal{A}_0$ can be defined as
\begin{equation}
f(a)=\frac{1}{2\pi {\rm i}}\int_\Gamma f(z)(zI-a)^{-1} dz,\label{integral_def}
\end{equation}
where the line integral of a Banach-space-valued function $g(z)$ defined on
a smooth curve $\gamma:t\mapsto z(t)$ for $t\in[0,1]$
is given by the norm limit of
Riemann sums of the form
$$
\sum_{j=1}^{\nu}
g(z(\theta_j))[z(t_j)-z(t_{j-1})],\qquad t_{j-1}\leq\theta_j\leq t_j,
$$
where $0=t_0 < t_1 < \ldots < t_{\nu -1} < t_{\nu} = 1$.

Denote by $\mathcal{H}(a)$ the algebra of analytic functions whose 
domain contains an open neighborhood of $\sigma(a)$.
The following well-known result is the basis for the
{\em holomorphic functional calculus}; see, e.g., \cite[page 206]{KRI}. 

\begin{theorem}
The mapping $\mathcal{H}(a)\longrightarrow\mathcal{A}_0$ defined by 
$f\mapsto f(a)$ is an algebra homomorphism, which maps the constant 
function $1$ to $I\in\mathcal{A}_0$ and maps the identity function to $a$.
If $f(z)=\sum_{j=0}^{\infty}c_jz^j$ is the power series representation 
of $f\in\mathcal{H}(a)$ over an open neighborhood of $\sigma(a)$, then we have   
$$f(a)=\sum_{j=0}^{\infty}c_ja^j.$$ 
\end{theorem}

Moreover, the following version of the \emph{spectral theorem} holds: 
 \begin{equation}
\sigma(f(a))=f(\sigma(a)).\label{eq_spectral_thm}
\end{equation}

If $a$ is normal, the following properties also hold:
\begin{itemize}
\item
$\|f(a)\|=\|f\|_{\infty,\sigma(a)} := \max_{\lambda \in \sigma (a)} |f(\lambda)|$;
\item
$\overline{f(a)}=[f(a)]^*$; in particular, if $a$ is Hermitian then 
$f(a)$ is also Hermitian if and only if $f(\sigma(a))\subset\mathbb{R}$;
\item
$f(a)$ is normal;
\item
$f(a)b=bf(a)$ whenever $b\in\mathcal{A}_0$ and $ab=ba$.
\end{itemize}

Obviously, these definitions and results apply in the case where
$a$ is a matrix $A$ with entries in a $C^*$-algebra 
$\mathcal{A}_0$. In particular, if $f(z)$ is analytic on a neighborhood of
$\sigma(A)$, we define $f(A)$ via
$$f(A) 
=\frac{1}{2\pi {\rm i}}\int_\Gamma f(z)(zI-A)^{-1} dz,\label{integral_def_A}
$$
with the obvious meaning of $\Gamma$.
    
\section{Bounds for the Hermitian case}\label{herm}

In this paper we will be concerned mostly with banded and
sparse matrices. A matrix $A\in \mathcal{A}$ 
is {\em banded with bandwidth $m$}  if $a_{ij}$ is the zero element 
of $\mathcal{A}_0$ whenever $|i-j|>m$. Banded matrices over 
a $C^*$-algebra arise in several contexts; see, e.g., 
\cite{Berezansky71,Berezansky98,CT06,Osipov} and references therein.  
     
Let $A\in \mathcal{A}$ be a banded Hermitian matrix with bandwidth $m$. 
In this section we provide exponentially decaying bounds on the norm
of the entries $[f(A)]_{ij}$ ($1\le i,j\le n$)
of $f(A)$, where $f$ is analytic on a neighborhood of $\sigma(A)$, and
we discuss the important case where the bounds do not depend on the 
order $n$ of the matrix. The results in this section extend to the
$C^*$-algebra setting analogous results for matrices over $\mathbb{R}$
or $\mathbb{C}$ found in \cite{BG99,BR07} and \cite{BBR13}.
Functions of non-Hermitian (and non-normal) matrices are studied
in the following sections.

Hermitian matrices have a real spectrum. If $\sigma(A)\subseteq
[\alpha, \beta]\subset \mathbb{R}$ then, by replacing $A$ (if necessary)
with the shifted and scaled matrix $\frac{2}{\beta - \alpha}A - 
\frac{\beta + \alpha}{\beta - \alpha} I_n$, we can assume that
the spectrum is contained in the interval ${\cal I} = [-1,1]$.
We also assume that $f(z)$ is real for real $z$, so that
$f$ maps Hermitian matrices to Hermitian matrices.
        
Let ${\mathbb P}_k$ denote the set of all complex
polynomials of degree at most $k$ on
$\cal I$. 
Given $p\in{\mathbb P}_k$, the matrix $p(A)\in \mathcal{A}$ is well 
defined and it is banded 
with bandwidth at most $km$. So for any polynomial $p\in{\mathbb P}_k$ 
and any pair of indices $i,j$ such that $|i-j|>km$ we have
    \begin{eqnarray}
    \|[f(A)]_{ij}\|&=&  \|[f(A)-p(A)]_{ij}\|\label{step1}\\
    &\leq&\|f(A)-p(A)\|\label{step2}\\
    &=&\rho(f(A)-p(A))\label{step3}\\
    &=&\sup(\sigma(f(A)-p(A)))=\sup(\sigma((f-p)(A)))\label{step4}\\
    &=&\sup((f-p)(\sigma(A)))\leq E_k(f),\label{step5}
    \end{eqnarray}
    where $E_k(f)$ is the best uniform approximation error for the function $f$ on 
    the interval $\cal I$ using polynomials of degree at most $k$: 
$$ E_k(f):= \min_{p\in {\mathbb P}_k} \max_{t\in {\cal I}} |f(t) - p(t)|\,.$$
    In the above computation:
    \begin{itemize}
    \item (\ref{step2}) follows from (\ref{step1}) as a consequence of the definition of operator norm,
    \item (\ref{step3}) follows from (\ref{step2}) because $A$ is Hermitian, so $f(A)-p(A)$ is also 
 Hermitian,
    \item the spectral theorem \eqref{eq_spectral_thm} allows us to obtain (\ref{step5}) from (\ref{step4}). 
    \end{itemize}

Next, we recall the classical Bernstein's Theorem concerning the 
asymptotic behavior 
of $ E_k(f) $ for $k\to \infty$; see, e.g., \cite[page 91]{meinardus}.
This theorem states that there exist constants $c_0>0$ and $0<\xi<1$ such that 
$E_k(f)\leq c_0\, \xi^{k+1}$. From this we can deduce exponentially decaying 
bounds for  $\|[f(A)]_{ij}\|$ with respect to $|i-j|$, 
by observing that $|i-j|>km$ implies  $k+1 < \frac{|i-j|}{m} + 1$ and therefore
\begin{equation}\label{decay_bound}
\| [f(A)]_{ij}\| \le c_0\, \xi^{\frac{|i-j|}{m} + 1} = c\, \zeta^{|i-j|},
\quad c = c_0\, \xi, \quad \zeta = \xi^{\frac{1}{m}} \in (0,1).
\end{equation}
  
The above bound warrants further discussion. Indeed, as it is stated it is
a trivial bound, in the sense that for any {\em fixed} matrix $A$ and function 
$f$ such that
$f(A)$ is defined one can always find constants $c_0 > 0$ and $0 < \xi < 1$
such that (\ref{decay_bound}) holds for all $i,j=1,\ldots ,n$; all one has to
do is pick $c_0$ large enough. Thus, the entries of $f(A)$ may exhibit no
actual decay behavior!
However, what is important here is that the constants $c_0$ and $\xi$ 
(or at least bounds for them) can be given explicitly in terms of  
properties of $f$ and, indirectly, in terms of the bounds $\alpha$ and
$\beta$ on the spectrum of $A$. If we have a sequence $\{A_n\}$ of $n\times n$
matrices such that 
\begin{itemize}
\item the $A_n$ are banded with bounded bandwidth (independent of $n$);
\item the spectra $\sigma(A_n)$ are all contained in a common interval
$\cal I$ (independent of $n$), say ${\cal I}=[-1,1]$,
\end{itemize}   
then the bound (\ref{decay_bound}) holds independent of $n$.
In particular, the entries of $f(A_n)$ will actually decay to zero away
from the main diagonal as $|i-j|$ and $n$ tend to infinity, at a rate 
that is uniformly bounded below  
by a positive constant independent of $n$.

More specifically, Bernstein's Theorem yields the values
$c_0=\frac{2 \chi M(f)}{\chi-1}$ and $\xi=1/\chi$, where $\chi$ is the 
sum of the semi-axes of an ellipse $\mathcal{E}_\chi$ with foci in $1$ 
and $-1$, such that $f(z)$ is continuous on $\mathcal{E}_\chi$ and analytic 
in the interior of $\mathcal{E}_\chi$ (and $f(z)\in\mathbb{R}$ whenever 
$z\in\mathbb{R}$); furthermore, we have set
$M(f)=\max_{z\in\mathcal{E}_\chi}|f(z)|$.  
    
Summarizing, we have established the following results:
\begin{theorem}\label{thm_h1}
Let $\mathcal A = {\mathcal A}_0^{n\times n}$ where ${\mathcal A}_0$
is a $C^*$-algebra and 
let $A\in\mathcal{A}$ be Hermitian with 
bandwidth $m$ and spectrum contained in $[-1,1]$. Let the complex 
function $f(z)$ be continuous on a Bernstein ellipse $\mathcal{E}_\chi$ 
and analytic in the interior of $\mathcal{E}_\chi$, and assume 
$f(z)\in \mathbb{R}$ for $z\in \mathbb{R}$. Then there exist constants $c>0$ 
and $0<\zeta<1$ such that
  $$
     \|[f(A)]_{ij}\|\leq c\,\zeta^{|i-j|} 
  $$ 
  for all $1\leq i,j\leq n$. Moreover, one can choose 
  $c=\max\left\{\|f(A)\|,
\frac{2 M(f)}{\chi-1}\right\}$ and $\zeta =\left(\frac{1}{\chi}\right)^{1/m}$.
  \end{theorem} 

\begin{remark}
Letting $\theta := - \ln \zeta >0$ the decay bound can be rewritten
in the form $\|[f(A)]_{ij}\|\leq c\, {\rm e}^{-\theta |i-j|}$, which is sometimes
more convenient. 
\end{remark}

\begin{theorem}\label{thm_h2}
Let ${\mathcal A}_0$ be a $C^*$-algebra and
let $\{A_n\}_{n\in\mathbb{N}}\subset {\mathcal A}_0^{n\times n}$ be a 
sequence of Hermitian matrices of increasing size, with bandwidths uniformly 
bounded by $m\in\mathbb{N}$  
 and spectra all contained in $[-1,1]$. Let the complex function 
$f(z)$ be continuous on a Bernstein ellipse $\mathcal{E}_\chi$ and analytic 
in the interior of $\mathcal{E}_\chi$, and assume 
$f(z)\in \mathbb{R}$ for $z\in \mathbb{R}$. Then there exist constants $c>0$ and 
$0<\zeta <1$, independent of $n$, such that
  $$
     \|[f(A_n)]_{ij}\|\leq c\,\zeta^{|i-j|} = c\, {\rm e}^{-\theta |i-j|}, \quad
     \theta = - \ln \zeta\,,
  $$ 
  for all indices $i,j$. Moreover, one can choose 
  $c=\max\left\{\|f(A)\|,
\frac{2 M(f)}{\chi-1}\right\}$ and $\zeta =\left(\frac{1}{\chi}\right)^{1/m}$.
  \end{theorem} 

\begin{remark}
It is worth noting that the decay bounds in the above results are actually
{\em families} of bounds; different choices of the ellipse $\mathcal{E}_\chi$
will result in different bounds. If $\chi$ and $\chi'$, with
$\chi < \chi'$,  are both admissible 
values, choosing $\chi'$ will result in a smaller value of $\zeta$, thus
yielding a faster asymptotic decay rate, but possibly a larger value of 
the prefactor $c$; in general, tighter bounds may be obtained by varying
$\chi$ for different values of $i$ and $j$. See \cite{BBR13} for examples
and additional discussion of this issue.
\end{remark}

\begin{remark}
The bounds in Theorem \ref{thm_h2} essentially state that as long as the
possible singularities of $f$ remain bounded away from the interval $[-1,1]$
or, slightly more generally, from a (fixed) interval $[\alpha, \beta]$ 
containing the union of all the spectra $\sigma (A_n)$, $n\in \mathbb{N}$, 
then the entries of $f(A_n)$ decay exponentially fast away from the main 
diagonal, at a rate 
that is bounded below by a positive constant that does not depend on $n$. 
As a 
consequence, for every $\varepsilon >0$ one can determine a bandwidth $M = M(\varepsilon)$
(independent of $n$) such that 
$$\|f(A_n) - [f(A_n)]_{M}\| < \varepsilon $$
holds for all $n$, where $[B]_{M}$ denotes the matrix with entries $b_{ij}$ equal to those
of $B$ for $|i-j|\le M$, zero otherwise. It is precisely this fact that makes
exponential decay an important property in applications; see, e.g.,
\cite{BBR13}.  
As a rule, the closer the possibile singularities of $f$ are to the spectral 
interval $[\alpha, \beta]$, the slower the
decay is (that is, the larger is $c$ and the closer $\zeta$ is to the upper bound $1$,
or $\theta$ to $0$).
\end{remark}

\begin{remark}
For entire functions, such as the exponential function $f(z) = {\rm e}^z$, 
the above exponential decay results are not optimal; indeed, in such cases
{\em superexponential} decay bounds can be established, exploiting the fact that
the coefficients in the Chebyshev expansion of $f$ decay superexponentially. 
For an
example of this type of result, see \cite{iserles}; see 
also Example \ref{example4} below.
Also, in some cases improved decay bounds can be obtained by using different
tools from polynomial approximation theory, or exploiting additional structure
in $f$ or in the spectra $\sigma (A_n)$; see \cite{BBR13}.
\end{remark}

\section{Bounds for the normal case}\label{normal}

We briefly discuss the case when the banded matrix $A\in \mathcal{A}$ 
is normal, but not necessarily Hermitian. As usual, we denote by $m$ the 
bandwidth of $A$. 

The main difference with respect to the previously discussed 
Hermitian case consists in the fact that $\sigma(A)$ is no longer 
real. Let ${\cal F}\subset\mathbb{C}$ be a compact, connected region 
containing $\sigma(A)$, and denote by $\mathbb{P}_k$, as before, 
the set of complex polynomials of degree at most $k$. Then the 
argument in (\ref{step1}-\ref{step5}) still holds, except that 
now polynomial approximation is no longer applied on a real interval, 
but on the complex region ${\cal F}$. Therefore, the following bound holds 
for all indices $i$, $j$ such that $|i-j|>km$:
\begin{equation}
\|[f(A)]_{ij}\|\leq\sup |(f-p)(\sigma (A))|\leq E_k(f,{\cal F}),
\label{bound_nonhermitian}
\end{equation}
where 
$$
E_k(f,{\cal F}):=\min_{p\in\mathbb{P}_k}\max_{z\in {\cal F}}|f(z)-p(z)|.
$$
Unless more accurate estimates for $\sigma(A)$ are available, 
a possible choice for $\cal F$ is the disk of center $0$ and radius 
$\rho(\hat{A})$: see Remark \ref{remark2}.

If $f$ is analytic on $\cal F$, bounds for $E_k(f,{\cal F})$ that decay 
exponentially with $k$ are available through the use of Faber 
polynomials: see \cite[Theorem 3.3]{BR07} and the next section for more details. 
More precisely, there exist constants $\tilde{c}>0$ and 
$0<\tilde{\lambda}<1$ such that $E_k(f,{\cal F})\leq \tilde c\, \tilde{\lambda}^k$ 
for all $k\in\mathbb{N}$. This result, together with 
\eqref{bound_nonhermitian}, yields for all $i$ and $j$ the bound
$$
\|[f(A)]_{ij}\|\leq c\, \lambda^{|i-j|} = c\, {\rm e}^{-\theta |i-j|}
$$
(where $\theta = -\ln \lambda$) for suitable constants $c>0$ and 
$0<\lambda<1$, which do not depend on 
$n$, although they generally depend on $f$ and $\cal F$.

\section{Bounds for the general case}\label{general}
 If $A$ is not normal, then the equality between (\ref{step2}) and 
(\ref{step3}) does not hold. 
We therefore need other explicit bounds
on the norm of a function of a matrix.
   
\subsection{The field of values and bounds for complex matrices}
Given a matrix $A\in\mathbb{C}^{n\times n}$, the associated {\em field 
of values} (or {\em numerical range}) is defined as
$$
W(A)=\left\{ \frac{x^*Ax}{x^*x}\,;\, x\in\mathbb{C}^n,\, x\neq 0\right\}.
$$
It is well known that $W(A)$ is a convex and compact subset of the 
complex plane that contains the eigenvalues of $A$.
    
The field of values of a complex matrix appears in the context of bounds 
for functions of matrices thanks to a result by Crouzeix (see \cite{crouzeix}):
\begin{theorem} \textnormal{(Crouzeix)}  \label{thm_crouzeix}
There is a universal constant $2\leq \mathcal{Q} \leq 11.08$ such that, 
given $A\in\mathbb{C}^{n,n}$, $\cal F$ a convex compact set containing the 
field of values $W(A)$, a function $g$ continuous 
on $\cal F$ and analytic in its interior, then the following inequality holds:
\begin{displaymath}
\|g(A)\|_2\leq \mathcal{Q}\,\sup_{z\in{\cal F}}|g(z)|.
\end{displaymath}
\end{theorem}

We mention
that Crouzeix has conjectured that ${\cal Q}$ can be replaced by 2,
but so far this has been proved only in some special cases. 

Next, we need to review some basic material on polynomial
approximation of analytic functions. Our treatment follows 
the discussion in \cite{BR07}, which in turn is based on
\cite{Mar77}; see
also \cite{Cur71,Sue64}. 
In the following, $\mathcal{F}$ denotes a continuum
containing more than one point.
By a {\em continuum} we mean a nonempty, compact and connected
subset of $\mathbb{C}$.
Let $G_{\infty}$ denote the component of the complement
of $\mathcal{F}$ containing the point at infinity.
Note that $G_{\infty}$
is a simply connected domain in the extended
complex plane $\overline{\mathbb{C}} = \mathbb{C} \cup \{\infty\}$.
By the Riemann Mapping Theorem there exists a function
$w = \Phi(z)$ which maps $G_\infty$ conformally onto
a domain of the form $|w| \, > \, \rho \, > \, 0$ satisfying the
normalization conditions
\begin{equation}\label{norm_cond}
        \Phi(\infty) = \infty, \quad \lim_{z\rightarrow \infty}\frac{\Phi(z)}{z} = 1;\end{equation}
$\rho$ is the {\em logarithmic capacity} of $\mathcal{F}$.
Given any integer $k > 0$, the function $[\Phi(z)]^k$
has a Laurent series expansion of the form
\begin{equation}
        [\Phi(z)]^k = z^k + \alpha_{k-1}^{(k)}z^{k-1} +\cdots+ \alpha_{0}^{(k)}
+\frac{\alpha_{-1}^{(k)}}{z} + \cdots \label{LS}
\end{equation}
at infinity.
The polynomials
\[
        \Phi_k(z) = z^k + \alpha_{k-1}^{(k)}z^{k-1} +\cdots+ \alpha_{0}^{(k)}
\]
consisting of the terms with nonnegative powers of $z$ in the expansion (\ref{LS}) are
called the {\em Faber polynomials} generated by the continuum $\mathcal{F}$.

Let $\Psi$ be the inverse of $\Phi$. By $C_R$ we denote the image under $\Psi$ of a circle
$|w|=R > \rho$. The (Jordan) region with boundary $C_R$ is
denoted by $I(C_R)$.
By \cite[Theorem 3.17, p.~109]{Mar77},
every function $f(z)$ analytic on $I(C_{R_0})$
 with $R_0 > \rho$
can be expanded in a series of Faber polynomials: 
\begin{equation}
        f(z) = \sum_{k=0}^\infty \alpha_k\Phi_k(z), \label{fs}
\end{equation}
where the series converges uniformly inside $I(C_{R_0})$.
The coefficients are given by
\[
\alpha_k = \frac{1}{2\pi {\rm i}}\int_{|w| = R} \frac{f(\Psi(w))}{w^{k+1}} dw
\]where $\rho < R < R_0$.
We denote the partial sums of the series in (\ref{fs}) by
\begin{equation} \label{poly}
        \Pi_k(z) := \sum_{i=0}^{k} \alpha_i\Phi_i(z).
\end{equation}
Each $\Pi_k(z)$ is a polynomial of degree at most $k$,
since each $\Phi_i(z)$ is of degree $i$.
We now recall a classical result that will be instrumental
in our proof of the decay bounds; for its proof see, e.g.,
\cite[Theorem 3.19]{Mar77}.

\begin{theorem}\textnormal{(Bernstein)} \label{BT}
Let $f$ be a function defined on $\mathcal{F}$. Then given
any $\varepsilon > 0$ and any integer
$k\geq 0$,
there exists a polynomial $\Pi_k$ of degree
at most $k$ and a positive constant
$c(\varepsilon)$
 such that
\begin{equation}
        |f(z) - \Pi_k(z)| \le c(\varepsilon)(q +
\varepsilon)^k \quad (0 < q < 1) \label{ineq}
\end{equation}
for all $z\in \mathcal{F}$ if and only if $f$
is analytic on the domain $I(C_{R_0})$, where
$R_0 = \rho/q$. In this case, the sequence
$\{\Pi_k\}$ converges uniformly to $f$ inside $I(C_{R_0})$
as $k\to \infty$.
\end{theorem}

Below we will make use of the
sufficiency part of Theorem~\ref{BT}. Note that the
choice of $q$ (with $0 < q < 1$) depends on the region where the function
$f$ is analytic. If $f$ is defined on a continuum
$\mathcal{F}$ with logarithmic capacity $\rho$ then we
can pick $q$ bounded away from $1$ as long as the function is analytic
on $I(C_{\rho/q})$. Therefore,
the rate of convergence is
directly related to the properties of the function $f$, such as the
location of its poles (if there are any).
For certain regions, in particular for the case of
convex $\mathcal{F}$, it is possible to obtain an explicit 
value for the constant $c(\varepsilon)$; see \cite{Ell83} 
and \cite[Section 3.7]{BR07} and
\cite[Section 2]{mastronardi} and the discussion following
Theorem \ref{thm_nn1} below.

We can then formulate the following result on the off-diagonal 
decay of functions of non-normal band matrices:

\begin{theorem}\label{Thm-NNB}
Let $A\in\mathbb{C}^{n\times n}$ be $m$-banded, and let $\mathcal{F}$
be a continuum containing $W(A)$ in its interior. 
Let $f$ be a function defined on $\mathcal{F}$ and assume that
$f$ is analytic on $I(C_{R_0})$ $(\supset W(A))$, with $R_0 = \frac{\rho}
{q}$ where $0 < q < 1$ and $\rho$ is the logarithmic capacity of
$\mathcal{F}$. Then there are constants  
$K>0$ and $0<\lambda<1$ such that
  $$
     \|[f(A)]_{ij}\|\leq K\,\lambda^{|i-j|}
  $$ 
  for all $1\leq i,j\leq n$. 
\end{theorem} 
\begin{proof}
Let $g=f-p_k$ in Theorem \ref{thm_crouzeix}, where $p_k(z)$ is a 
polynomial of degree smaller than or equal to $k$.
Then $p_k(A)$ 
is a banded matrix with bandwidth at most $km$. Therefore, for all 
$i,j$ such that $|i-j|>km$ we have
$$|[f(A)]_{ij}|= |[f(A)]_{ij}-[p_k(A)]_{ij}|\leq \|f(A)-p_k(A)\|_2
\leq \mathcal{Q}\sup _{z\in \mathcal{F}} |f(z)-p_k(z)|.
$$
Now, by Theorem \ref{BT} we have that for any $\varepsilon > 0$
there exists a sequence of polynomials $\Pi_k$ of degree $k$ which
satisfy for all $z\in \mathcal{F}$
$$|f(z) - \Pi_k (z)| \le c(\varepsilon) (q + \varepsilon)^k, \quad
\textnormal{where} \quad 0 < q < 1\,.$$
Therefore, taking $p_k = \Pi_k$ and applying Theorem \ref{BT} we obtain 
$$
|[f(A)]_{ij}|\leq\mathcal{Q}\,c(\varepsilon)\, (q+\varepsilon)^{\frac{|i-j|}{m}}.
$$
The thesis follows if we take $\lambda = (q+\varepsilon)^{\frac{1}{m}} < 1$
and $K=\max\left\{\|f(A)\|_2,\mathcal{Q}\,c(\varepsilon) \right\}$.
\end{proof}

We mention that a similar result (for the case of multi-band matrices)
can be found in
\cite[Theorem 2.6]{mastronardi}.\\

The assumptions in Theorem \ref{Thm-NNB} are fairly general.
In particular, the result applies
if $f(z)$
is an entire function; in this case, however, better estimates exist
(for instance, in the case of the matrix exponential; see, e.g.,
\cite{BR09} and references therein).  

The main difficulty in applying the theorem to obtain practical
decay bounds is in estimating the constant $c(\varepsilon)$ and 
the value of $q$, which requires knowledge of the field of values
of $A$ (or an estimate of it) and of the logarithmic capacity 
of the continuum $\mathcal{F}$ containing $W(A)$. The task is
made simpler if we assume (as it is is natural) that $\mathcal{F}$
is convex; see the discussion in \cite{BR07}, especially Theorem 3.7.
See also \cite[Section 2]{mastronardi} and the next
subsection for further discussion.

The bound in Theorem \ref{Thm-NNB} often improves on 
previous bounds for diagonalizable matrices in \cite{BR07} 
containing the condition number of the eigenvector matrix,  
especially when the latter is ill-conditioned (these bounds
have no analogue in the $C^*$-algebra setting). 

Again, as stated, Theorem \ref{Thm-NNB} is non-trivial only if $K$ and $\lambda$
are independent of $n$. We return on this topic in the next subsection.\\

It is worth noting that since $A$ is not assumed to have 
symmetric structure,
it  could have different numbers of nonzero
diagonals below and above the main diagonal.
Thus, it may be desirable to have bounds that 
account for the fact that in such cases the rate
of decay will be generally different above and below 
the main diagonal. An extreme case is when
$A$ is an upper (lower) Hessenberg matrix, in which case
$f(A)$ typically exhibits fast decay below (above) the
main diagonal, and generally no decay above (below) it.

For diagonalizable matrices over $\mathbb{C}$, such a 
a result can be found in \cite[Theorem 3.5]{BR07}. 
Here we state an analogous result without the diagonalizability
assumption. We say that a matrix $A$ has {\em lower bandwidth}
$p > 0$ if $a_{ij} = 0$ whenever $i - j> p$
and {\it upper bandwidth} $s > 0$
if $a_{ij} = 0$ whenever $j - i > s$.
We note that if $A$ has lower bandwidth
$p$ then $A^k$ has lower bandwidth $kp$
for $k = 0,1,2,\dots$, and
similarly for the upper bandwidth $s$.
Combining the argument found in the proof of \cite[Theorem 3.5]{BR07}
with Theorem \ref{Thm-NNB}, we obtain the following result.

\begin{theorem} \label{bineq}
Let $A\in \mathbb{C}^{n\times n}$ be a matrix with lower bandwidth
$p$ and upper bandwidth $s$, and let the function $f$ satisfy the
assumptions of Theorem \ref{Thm-NNB}.
Then there exist constants $K>0$ and
 $0 < \lambda_1,\lambda_2 < 1$ such that for $i \geq j$
\begin{equation}
        |[f(A)]_{ij}| < K\,\lambda_1^{i-j} \label{lineq}
\end{equation}
and for $i < j$
\begin{equation}
        |[f(A)]_{ij}| < K\, \lambda_2^{j-i}. \label{uineq}
\end{equation}
\end{theorem}

The constants $\lambda_1$ and $\lambda_2$ depend on the position
of he poles of $f$ relative to the continuum $\mathcal{F}$; they
also depend, respectively, on
the lower and upper bandwidths $p$
and $s$ of $A$. 
For an upper Hessenberg matrix
($p=1$, $s=n$) only the bound (\ref{lineq}) 
is of interest,
particularly in the situation (important in applications) where
we consider sequences of matrices of increasing size.
Similarly, for a lower Hessenberg matrix ($s=1$, $p=n$)
only (\ref{uineq}) is meaningful.
More generally, the bounds are of interest when they are applied 
to sequences of $n\times n$ matrices $\{A_n\}$ for which either
$p$ or $s$ (or both) are fixed as $n$ increases, and such that there is a
fixed connected compact set $\mathcal{F}\subset \mathbb{C}$ containing
$W(A_n)$ for all $n$ and excluding the singularities of $f$
(if any). In this case the relevant constants 
in Theorem \ref{bineq} are independent of $n$, and we obtain
uniform exponential decay bounds.

Next, we seek to generalize Theorem \ref{Thm-NNB}
to the $C^*$-algebra setting. In order to do this, we need
to make some preliminary observations.
 If $T$ is a bounded linear operator on a Hilbert space 
$\mathcal{H}$, then its numerical range is defined as 
$W(T)=\{\langle Tx,x\rangle\,;\, x\in \mathcal{H}\,, \|x\|=1\}$. 
The generalization 
of the notion of numerical range
to $C^*$-algebras (see \cite{berberianorland}) is formulated via 
the Gelfand--Naimark representation: $a\in \mathcal{A}_0$ is 
associated with an operator $T_a$ defined on a suitable Hilbert 
space. Then $\overline{W(T_a)}$, the closure of $W(T_a)$, does not 
depend on the particular $*-$representation that has been chosen for 
$\mathcal{A}_0$. In other words, the closure of the numerical range is 
well defined for elements of $C^*$-algebras  
(whereas the numerical range itself, in general, is not). 
This applies, in particular, to elements of the $C^*$-algebra $\mathcal{A}=
\mathcal{A}_0^{n\times n}$.
 
Let us now consider a matrix $A\in \mathcal{A}$. 
In the following, we will need easily computable bounds on $\overline{W(A)}$. 
Theorem \ref{thm_matricial_norm} easily implies the following simple
result: 
 \begin{proposition}\label{prop_fov_inside_disk}
 Let $A\in\mathcal{A}$. Then $\overline{W(A)}$ is contained in the disk of center $0$ and radius $\|\hat{A}\|_2$.
 \end{proposition}
 
We are now in a position to derive bounds valid in the general,
nonnormal case.

\subsection{Bounds for the nonnormal case}
   
 Our aim is to extend the results in the previous section
 to the case where $A$ is a 
matrix over a $C^*$-algebra. In \cite{crouzeix}, Crouzeix provides a 
useful generalization of his result from complex matrices to 
bounded linear operators on a Hilbert space $\mathcal{H}$. 
Given a set $E\subset 
\mathbb{C}$, denote by $\mathcal{H}_b(E)$ the algebra of continuous 
and bounded functions in $\overline{E}$ which are analytic in the 
interior of $E$. Furthermore, for $T\in \mathcal{B}(\mathcal{H})$ 
let $\|p\|_{\infty,T}:= \sup_{z\in \overline{W(T)}} |p(z)|$.
 Then we have (\cite{crouzeix}, Theorem 2):
\begin{theorem}\label{thm_crouseix2}
For any bounded linear operator 
$T\in\mathcal{B}(\mathcal{H})$ the homomorphism $p\mapsto p(T)$ from 
the algebra $\mathbb{C}[z]$, with norm $\|\cdot\|_{\infty,T}$, 
into the algebra $\mathcal{B}(\mathcal{H})$, 
is bounded with constant $\mathcal{Q}$. 
It admits a unique bounded extension from $\mathcal{H}_b(W(T))$ into 
$\mathcal{B}(\mathcal{H})$. This extension is bounded with constant $\mathcal{Q}$.
   \end{theorem}
  
Using again the Gelfand--Naimark representation together with the
notion of numerical range for elements of $\cal A$, we obtain
as a consequence:
  
\begin{corollary}\label{corollary1}
Given $A\in\mathcal{A}$, the following bound holds 
for any complex function $g$ analytic on a neighborhood of
$\overline{W(A)}$:
$$
  \|g(A)\|\leq {\cal Q} \|g\|_{\infty,A}=
{\cal Q} \sup_{z\in \overline{W(A)}}|g(z)|.
  $$
  \end{corollary}
  
Since we wish to obtain bounds on $\|[f(A)]_{ij}\|$, where the function 
$f(z)$ can be assumed to be analytic on an open set $S\supset\overline{W(A)}$, 
we can choose $g(z)$ in Corollary \ref{corollary1} as $f(z)-p_k(z)$, 
where $p_k(z)$ is any complex polynomial of degree bounded by $k$. 
The argument in (\ref{step1})--(\ref{step5}) can then be adapted as follows:
  \begin{eqnarray}
    \|[f(A)]_{ij}\|&=&  \|[f(A)-p_k(A)]_{ij}\|\label{step1nn}\\
    &\leq&\|f(A)-p_k(A)\|\label{step2nn}\\
    &\leq& {\cal Q}\, \|f-p_k\|_{\infty,A}\label{step3nn}\\
    &=&{\cal Q} \sup_{z\in \overline{W(A)}}|f(z)-p_k(z)|\label{step4nn}\\
    &\leq& {\cal Q}\, E_k(f,\overline{W(A)}),\label{step5nn}
    \end{eqnarray}
  where $E_k(f,\overline{W(A)})$ is the degree $k$ best approximation 
error for $f$ on the compact set $\overline{W(A)}$. In order to make explicit 
computations easier, we may of course replace  $\overline{W(A)}$ with 
a larger but more manageable set in the above argument, 
as long as the approximation theory results used in the proof
of Theorem \ref{Thm-NNB} can be applied.

  \begin{theorem}\label{thm_nn1}
Let $A\in\mathcal{A}$ be an $n\times n$ matrix of bandwidth $m$ 
and let the function $f$ and the continuum $\mathcal{F}\supset
\overline{W(A)}$ satisfy the assumptions of Theorem \ref{Thm-NNB}.
Then there exist explicitly computable constants $K>0$ and $0<\lambda<1$ such 
that
  $$
     \|[f(A)]_{ij}\|\leq K\lambda^{|i-j|}
  $$ 
  for all $1\leq i,j\leq n$. 
    \end{theorem} 
  
A simple approach to the computation of constants $K$ and $\lambda$ 
goes as follows. It follows from Proposition \ref{prop_fov_inside_disk} 
that the set $\mathcal{F}$ in Theorem \ref{thm_nn1} can be chosen as 
the disk of center $0$ and radius $r=\|\hat{A}\|_2$. Assume that $f(z)$ 
is analytic on an open neighborhood of the disk of center $0$ and 
radius $R>r$. The standard theory of complex Taylor 
series
gives the following estimate for the Taylor approximation error 
\cite[Corollary 2.2]{Ell83}:
\begin{equation}
E_k(f,\mathcal{C})\leq\frac{M(R)}{1-\frac{r}{R}}\left(\frac{r}{R}\right)^{k+1},
\label{eq_taylor}
\end{equation}
where $M(R)=\max_{|z|=R}|f(z)|$.
Therefore we can choose 
  $$K=\max\left\{\|f(A)\|,\,\mathcal{Q}\,M(R)\frac{r}{R-r}\right\},
\qquad\lambda=\left(\frac{r}{R}\right)^{1/m}.$$

The choice of the parameter $R$ in \eqref{eq_taylor} is somewhat 
arbitrary: any value of $R$ will do, as long as $r<R<\min|\zeta|$, 
where $\zeta$ varies over the poles of $f$ (if $f$ is entire,
we let $\min|\zeta|=\infty$). Choosing as large a
value of $R$ as possible gives a better asymptotic decay rate, 
but also a potentially 
large constant $K$. For practical purposes, one may therefore want to 
pick a value of $R$ that ensures a good trade-off between the magnitude 
of $K$ and the decay rate: see the related discussion in  
\cite{BR07} and \cite{BBR13}.

As in the previous section, we are also interested in the case of a 
sequence $\{A_n\}_{n\in\mathbb{N}}$ of matrices of increasing size 
over $\mathcal{A}_0$. In order to obtain a uniform bound, we reformulate 
Corollary \ref{corollary1} as follows.
  
\begin{corollary}\label{corollary2}
Let $\{A_n\}_{n\in\mathbb{N}}$ be a sequence of $n\times n$ matrices
over a $C^*$-algebra $\mathcal{A}_0$ such that there exists a 
connected compact set ${\cal C}\subset\mathbb{C}$ that contains 
$\overline{W(A_n)}$ for all $n$, and let 
$g$ be a complex function analytic on a neighborhood of
$\cal C$. The following uniform bound holds:
  $$
  \|g(A_n)\|\leq \mathcal{Q} \|g\|_{\infty,{\cal C}}=
\mathcal{Q}\sup_{z\in {\cal C}}|g(z)|.
  $$
  \end{corollary}
  
We then have a version of Theorem \ref{thm_nn1} for matrix sequences 
having uniformly bounded bandwidths and fields of values:
\begin{theorem}\label{thm_nn2}
Let $\{A_n\}_{n\in\mathbb{N}}\subset\mathcal{A}$ be a sequence of 
$n\times n$ matrices over a $C^*$-algebra $\mathcal{A}_0$ with bandwidths uniformly
bounded by $m$. Let the complex 
function $f(z)$ be analytic on a neighborhood of a connected compact set 
$\mathcal{C}\subset\mathbb{C}$ containing $\overline{W(A_n)}$ for 
all $n$. Then there exist explicitly computable constants $K>0$ and 
$0<\lambda<1$, independent of $n$, such that
$$
     \|[f(A_n)]_{ij}\|\leq \mathcal{Q} \, E_k(f,\mathcal{C})\leq K\lambda^k
$$ 
for all indices $i,j$. 
\end{theorem}

In other words: as long as the singularities of $f$ (if any) 
stay bounded away from a fixed compact set $\mathcal{C}$ containing
the union of all the sets $\overline{W(A_n)}$, and as long as 
the matrices $A_n$ have bandwidths less than a fixed integer
$m$, the entries of $f(A_n)$ decay exponentially fast
away from the main diagonal, at a rate bounded below by a 
fixed positive constant as $n\to \infty$. The larger the distance between the
singularities of $f$ and the compact $\mathcal{C}$, the larger
this constant is.


Finally, it is straightforward to generalize Theorem \ref{bineq} to the case
of matrices over a general $C^*$-algebra.  
This completes the desired extension to the $C^*$-algebra
setting of the known exponential decay results for analytic functions of
banded matrices over $\mathbb{R}$ or $\mathbb{C}$.


\section{The case of quaternion matrices} \label{quater}
Matrices over the real division algebra $\mathbb{H}$ of quaternions 
have many interesting properties; see, e.g., \cite{Loring,Zhang} and 
the references therein, as well as \cite{LeBihan} for a recent application
of quaternion matrices to signal processing. There is, however, very
little in the literature about functions of matrices over $\mathbb{H}$,
except for the very special case of the inverse $f(A)=A^{-1}$
of an invertible quaternion matrix. 
This is no doubt due to the fact that fundamental difficulties arise
even in the scalar ($n=1$) case when attempting to extend the classical
theory of complex analytic functions to functions of a quaternion
variable \cite{Deavours,Sudbery}. 

The formal evaluation of functions of matrices with quaternion entries was
considered by Giscard and coworkers in \cite{giscard}, the same paper 
that raised the question 
that led to the present work. In order to even state a
meaningful generalization of our decay results 
for analytic functions of banded (or sparse)
matrices over $\mathbb{R}$ or $\mathbb{C}$ to matrices over $\mathbb{H}$,
we first need to restrict the class of functions under consideration 
to those analytic functions that can be expressed by convergent power
series with {\em real} coefficients. In this case, no ambiguity can
arise when considering functions of a matrix of the form
$$f(A) = a_0 I_n + a_1 A + a_2 A^2 + \cdots + a_k A^k +\cdots , 
\quad  A\in \mathbb{H}^{n\times n},$$
since the real field $\mathbb{R}$ is the center of the quaternion
algebra $\mathbb{H}$. In contrast, for functions expressed by power
series with (say) complex coefficients we would have to distinguish between
\lq\lq left\rq\rq\ and \lq\lq right\rq\rq\ power
series, since $a_k A^k\neq A^k a_k$ in general.
Fortunately, many of the most important functions (like the exponential, the
logarithm, the trigonometric and hyperbolic functions and their inverses, etc.)
can be represented by power series with real coefficients.

Next, we note that the quaternion algebra $\mathbb{H}$ is not a
$C^*$-algebra; first of all, it's a real algebra (not a complex one), and 
second, it is a noncommutative division algebra. The Gelfand--Mazur
Theorem \cite{MC} states that a $C^*$-algebra which is a division 
algebra is
$*$-isomorphic to $\mathbb{C}$ and thus it is necessarily commutative.
Hence, we cannot immediately apply the results from the previous
sections to functions of quaternion matrices.

To obtain the desired generalization, we make use of the fact that  
quaternions can be regarded as 
$2\times 2$ matrices over $\mathbb{C}$  
through the following representation:
$$
\mathbb{H}=\left\{q=a+b\,{\rm i}+c\,{\rm j}+d\,{\rm k}; 
a,b,c,d\in\mathbb{R}\right\}\cong
\left\{Q=\left(\begin{array}{cc}
a+b\,{\rm i} & c+d\,{\rm i}\\
-c+d\,{\rm i} & a-b\,{\rm i}
\end{array}\right)\right\}.
$$
The modulus (or norm) of a quaternion is given by $|q|=\sqrt{a^2+b^2+c^2+d^2}=
\|Q\|_2$, where $Q$ is the matrix associated with $q$.

Thus, we represent matrices over quaternions as complex block 
matrices with blocks of size $2\times 2$. In this way the {\rm real} algebra
$\mathbb{H}^{n\times n}$ with the natural operator norm
$$\|A\| = \sup_{x\ne 0} \frac{\|Ax\|}{\|x\|}, 
\quad x=(x_1,\ldots ,x_n)\in \mathbb{H}^n,
\quad \|x\| = \left (\sum_{i=1}^n |x_i|^2\right )^{\frac{1}{2}},$$
is isomorphic to a norm-closed real subalgebra 
$\mathcal{B}$ of the
$C^*$-algebra $\mathcal{A} = \mathbb{C}^{2n\times 2n}$. 
The operator norm of an $n\times n$ quaternion matrix $A$ 
turns out to coincide with the spectral norm
of the $2n\times 2n$ complex matrix $\varphi (A)$ that corresponds to $A$ in this
representation: $\|A\| = \|\varphi (A)\|_2$  
(see \cite[Theorem 4.1]{Loring}).

Let now $f$ be a function that can be expressed by a power series
$f(z) = \sum_{k=0}^\infty a_kz^k$ with $a_k\in \mathbb{R}$, and assume 
that the power series has radius of convergence $R>\|A\| = \|\varphi (A)\|_2$.
Then the function $f(A)$ is well defined,\footnote{Since the subalgebra
of the $C^*$-algebra $\mathbb{C}^{2n\times 2n}$ that corresponds via
$\varphi$ to $\mathbb{H}^{n\times n}$ is closed under linear combinations
with real coefficients and norm-closed, the matrix $f(A)$ is a well-defined
quaternion matrix that satisfies
$\varphi (f(A)) = f(\varphi (A))$.}
and is given by the convergent
power series $f(A) = \sum_{k=0}^\infty a_kA^k$.

The theory developed in sections \ref{herm}-\ref{general} can now be
applied to obtain the desired exponential decay bounds for functions
of banded quaternion matrices, at least for 
those analytic functions that
can be expressed by convergent power series with real coefficients.

\section{General sparsity patterns}

Following \cite{BR07} and \cite{BBR13}, we sketch an adaptation of 
Theorems \ref{thm_h1} and \ref{thm_nn1} to the case where the $n\times n$ 
matrix $A\in\mathcal{A}$ is not necessarily banded, but it has a more 
general sparsity pattern. 

Recall that the graph $G_A$ associated with $A$ is defined as follows:
\begin{itemize}
\item
$G_A$ has $n$ nodes,
\item
nodes $i$ and $j$ are connected by an edge if and only if $a_{ij}\neq 0$.
\end{itemize}
The \emph{geodetic distance} $d(i,j)$ between nodes $i$ and $j$ is 
the length of the shortest path connecting node $i$ to node $j$. If 
there is no path from $i$ to $j$, then we set $d(i,j)=\infty$. Observe 
that in general $d(i,j)\neq d(j,i)$, unless the sparsity pattern of 
$A$ is symmetric.

Also recall that the \emph{degree} of a node $i$ is the number of 
nodes of $G_A$ that are directly connected by an edge to node $i$, that is, 
the number of neighbors of node $i$. It is equal to the number of 
nonzero entries in the $i$-th row of $A$. 

Let $a_{ij}^{(k)}$ be the $(i,j)$-th entry of the matrix $A^k$. It 
can be proved that $a_{ij}^{(k)}=0$ whenever $d(i,j)>k$, for all positive 
integers $k$. In particular, if  $d(i,j)>k$ then the $(i,j)$-th entry of 
$p_k(A)$ is zero, for any polynomial $p_k(z)$ of degree bounded by $k$. 
Therefore, equations \eqref{step1} and \eqref{step1nn} still hold if 
the condition $|i-j|>km$ is replaced by $d(i,j)>k$. Bounds for 
$\|[f(A)]_{ij}\|$ are then obtained in a straightforward way: we have
$$
\|[f(A)]_{ij}\|\leq c\,\xi^{d(i,j)}
$$
for the Hermitian case, and
$$
\|[f(A)]_{ij}\|\leq K\,\lambda^{d(i,j)}
$$
for the non-Hermitian case, where the constants $c,\xi,K,\eta$ are 
the same as in Theorems \ref{thm_h1} and \ref{thm_nn1} and their proofs.

Results for functions of sequences of matrices (Theorems \ref{thm_h2} 
and \ref{thm_nn2}) can also be adapted to general sparsity patterns in 
a similar way. Note that the hypothesis that the matrices $A_n$ have 
uniformly bounded bandwidth should be replaced by the condition that 
the degree of each node of the graph associated with $A_n$ should be 
uniformly bounded by a constant independent of $n$. 

\section{Examples}

In this section we show the results of some experiments on the decay 
behavior of $f(A)$ for various choices of $f$ and $A$
and comparisons with \emph{a priori} decay bounds.
We consider matrices over commutative $C^*$-algebras
of continuous functions, block matrices,
and matrices over the noncommutative 
quaternion algebra. 

\subsection{Matrices over $C([a,b])$}

Here we consider simple examples of matrices over $C([a,b])$, 
the algebra of (complex-valued)
continuous functions defined on a closed real interval $[a,b]$. 

Let $A$ be such a matrix: each entry of $A$ can be written as 
$a_{ij}=a_{ij}(t)$, where $a_{ij}(t)\in C([a,b])$. Let $f(z)$ be 
a complex analytic function such that $f(A)$ is well defined. 
In order to compute 
$f(A)$ we consider two approaches.
\begin{enumerate}
\item
A symbolic (exact) approach, based on the integral definition 
\eqref{integral_def}. This approach goes as follows:
\begin{itemize}
\item
Assuming $z\notin \sigma(A)$, compute symbolically 
$M=f(z)(zI-A)^{-1}$. Recall that the entries 
of $M$ are meromorphic functions of $t$ and $z$. In particular, 
if $A$ is invertible the inverse $B=A^{-1}$ can be computed 
symbolically, and its entries are elements of $C([a,b])$.
\item
Compute ${\rm det} (M)$ and factorize it as a polynomial in $z$. 
The poles of the entries of $M$ are  roots of ${\rm det} (M)$ with 
respect to the variable $z$.
\item 
Apply the residue theorem: $[f(M)]_{ij}$ is the sum of the residues 
of $M_{ij}$ at the roots of ${\rm det} (M)$. Such residues can be 
computed via a Laurent series expansion: see for instance the Maple 
commands {\tt series} and {\tt residue}.
\end{itemize}
The norms $\|[f(A)]_{ij}\|_{\infty}$ can be computed symbolically 
(see for instance the Maple command {\tt maximize}) or numerically 
via standard optimization methods. 
The exact approach is rather time-consuming and can only be applied 
to moderate-sized matrices. 
\item
An approximate hybrid (numerical-symbolic) approach, 
based on polynomial approximation of 
$f(z)$. In the present work we employ the following technique:
\begin{itemize}
\item 
Compute the coefficients of the Chebyshev approximating polynomial 
$p(z)$ that approximates $f(z)$ up to a predetermined 
degree or tolerance. Here we use the function {\tt chebpoly} of the 
{\tt chebfun} package \cite{chebfun} for Matlab.
If necessary, scale and shift $A$ so that its spectrum is contained 
in $[-1,1]$.
\item
Symbolically compute $f(A)\approx p(A)$. 
\end{itemize}
This approach gives results that are virtually indistinguishable
from the exact (purely symbolic) approach, but it is much more efficient
and can be applied to relatively large matrices.
\end{enumerate}

\begin{example}\label{example1}
Let $C$ be the following bidiagonal Toeplitz matrix of size $n\times n$ 
over $C([1,2])$:
$$
C=\left[
\begin{array}{llll}
1 & {\rm e}^{-t} &    & \\
 & 1 & \ddots   & \\
 &  & \ddots &  {\rm e}^{-t} \\
 &  &  &1
\end{array}
\right]
$$

Obviously, $C$ has an inverse in $C([1,2])$, which can be expressed as
a (finite) Neumann series.
We compute $C^{-1}$ symbolically, using the Symbolic Math 
Toolbox of Matlab, and then we compute the $\infty$-norms of its elements 
using the Matlab function {\tt fminbnd}. Figure \ref{fig_ex1} shows the 
corresponding mesh plot of the matrix $[\|b_{ij}\|_{\infty}]$ with $B=C^{-1}$
for $n=20$. Note the rapid off-diagonal decay.
\end{example}

\begin{example}\label{example2}
Let $A=CC^T$, where $C$ is defined as in Example \ref{example1}. The 
inverse of $A$ can be computed symbolically as $A^{-1}=C^{-T}C^{-1}$. 
Figure  \ref{fig_ex2} shows the mesh plot of the matrix of infinity norms of 
elements of $A^{-1}$ for $n=20$.
\end{example}

Next we consider the matrix exponential.

\begin{example}\label{example3}
Let $A$ be a tridiagonal Toeplitz Hermitian matrix as in Example 
\ref{example2}. 
We first scale it so that its spectrum is contained in $[-1,1]$. 
This is done by replacing $A$ with $A/\|\hat{A}\|_2$, where $\hat{A}$ 
is the matrix of infinity norms of the entries of $A$.
Next, we compute an approximation of the exponential of $A$ as 
$${\rm e}^A\approx\sum_{j=0}^k c_jT_j(A),$$ where the coefficients 
$\{c_j\}_{j=0,\dots ,k}$ are computed numerically using the 
{\tt chebpoly} function of Chebfun \cite{chebfun}, 
and the matrices $T_j(A)$ are 
computed symbolically using the Chebyshev recurrence relation. 
Here we choose $n=20$ and $k=8$.
 See Figure \ref{fig_ex3} for the mesh plot of the matrix of norms of 
elements of ${\rm e}^A$.

Observe that $\|\sum_{j=0}^kc_jT_j(A)\|_{\infty}\leq\sum_{j=0}^k|c_j|$, 
so $|c_k|$ gives an estimate of the correction to the approximation 
due to the highest order term (see also \cite[Section 4.1]{BR07}). 
If this correction is sufficiently small, we can assume that the 
Chebyshev approximation is accurate enough for our purposes. In this 
example we have $c_{8}= 1.9921\cdot 10^{-7}$ and $c_9=1.1037\cdot 10^{-8}$.
\end{example}

\begin{figure}
\begin{center}
\includegraphics[width=0.7\textwidth]{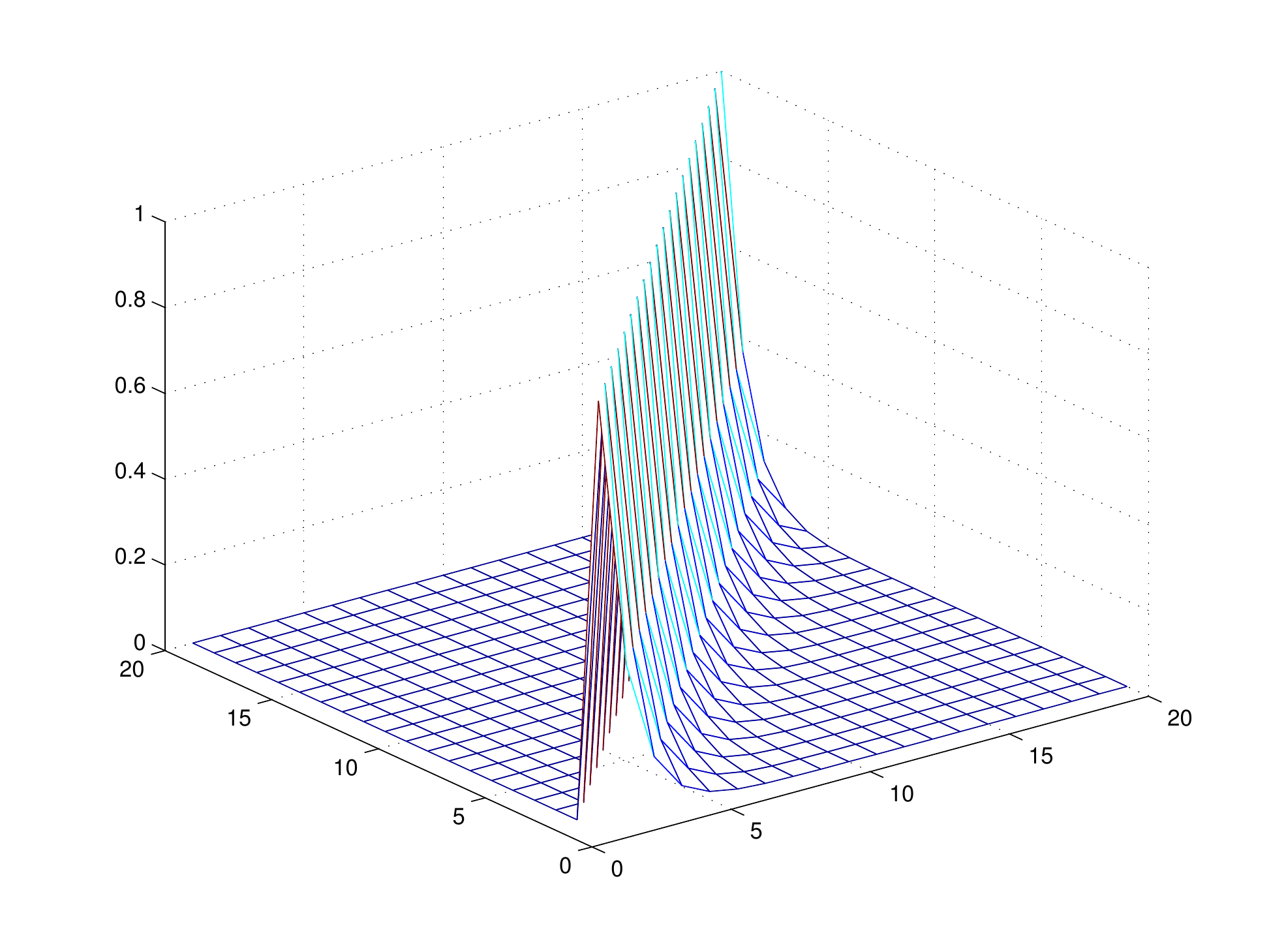}
\caption{Decay behavior for the inverse of the bidiagonal matrix 
in Example \ref{example1}.}\label{fig_ex1}
\end{center}
\end{figure}

\begin{figure}
\begin{center}
\includegraphics[width=0.7\textwidth]{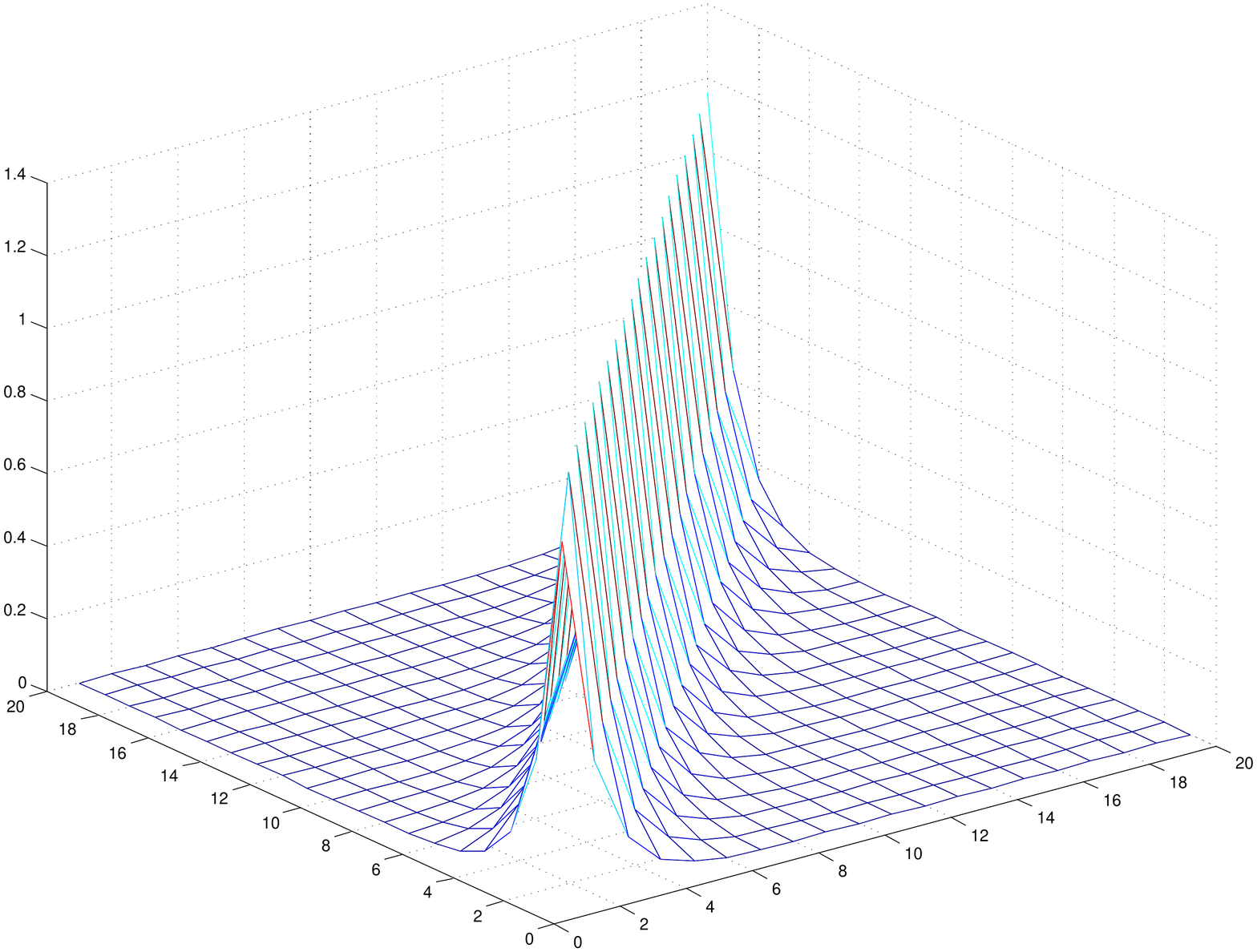}
\caption{Decay behavior for the inverse of the tridiagonal matrix 
in Example \ref{example2}.}\label{fig_ex2}
\end{center}
\end{figure}

\begin{figure}
\begin{center}
\includegraphics[width=0.7\textwidth]{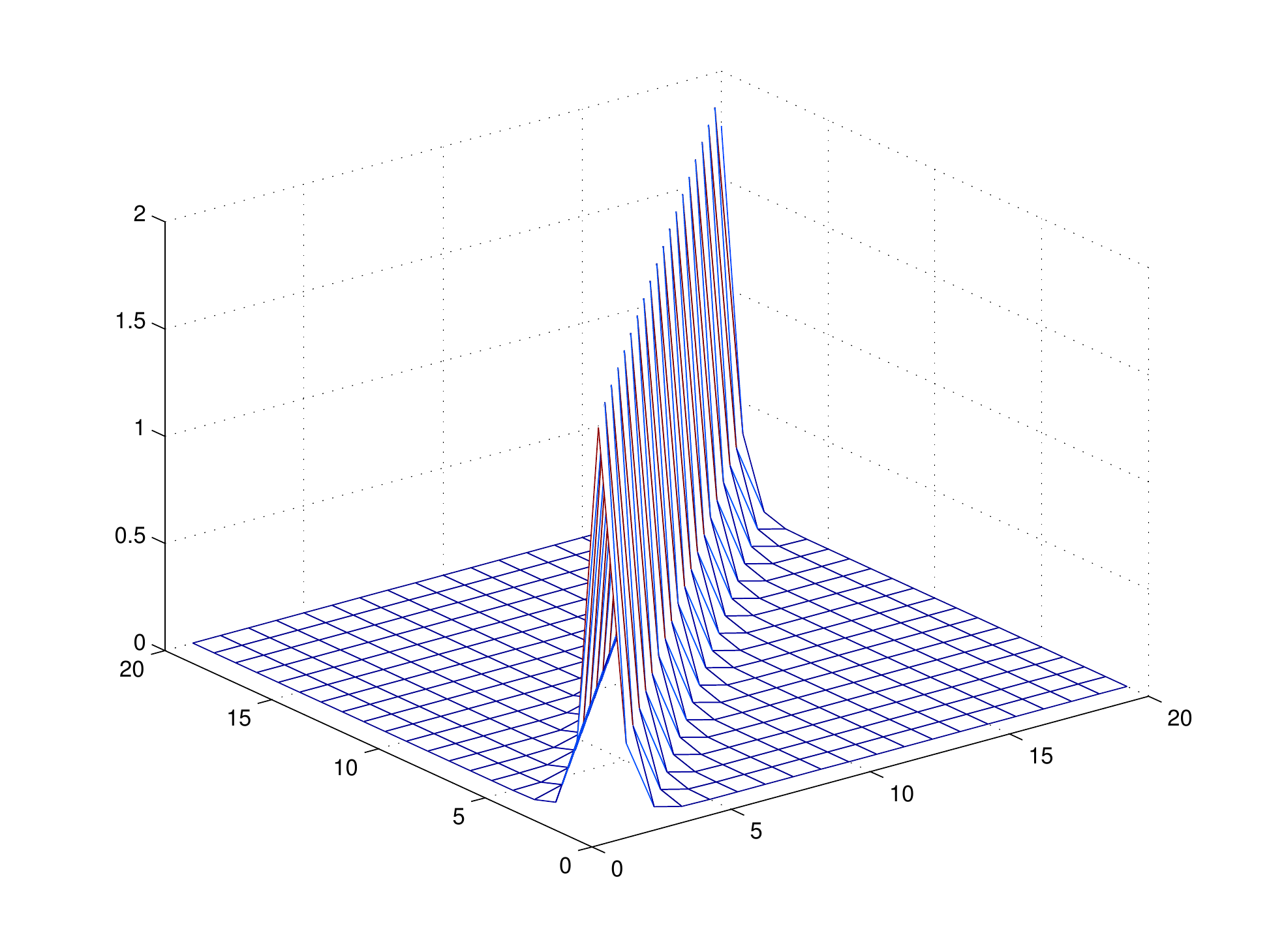}
\caption{Decay behavior for the exponential of the scaled 
tridiagonal matrix in Example \ref{example3}.}\label{fig_ex3}
\end{center}
\end{figure}

\begin{example}\label{example4}
Consider the tridiagonal Hermitian Toeplitz matrix of size 
$20\times 20$ over $C([0,1])$:
$$
A=\left[
\begin{array}{llll}
1 & {\rm e}^{-t} &    & \\
 {\rm e}^{-t} &1 & \ddots   & \\
 &  \ddots& \ddots &  {\rm e}^{-t} \\
 &  &  {\rm e}^{-t}&1
\end{array}
\right].
$$
We scale $A$ so that $\sigma(A)\subset [-1,1]$ and then
compute the Chebyshev approximation ${\rm e}^A\approx\sum_{j=0}^{12}c_j A^j$.
The approximation error is bounded in norm by $3.9913\cdot 10^{-14}$.
The decay behavior of ${\rm e}^A$ and the comparison with decay 
bounds for different choices of $\chi$ (cf.~Theorem \ref{thm_h2})
are shown in Figures \ref{fig_ex4a} and \ref{fig_ex4b}. 
The semi-logarithmic
plot clearly shows the superexponential decay in the entries
of the first row of ${\rm e}^A$, which is to be expected since
the coefficients $c_k$ in the Chebyshev expansion of ${\rm e}^z$
decay faster than exponentially as $k\to \infty$ \cite[page 96]{meinardus}.
In contrast, our bounds, being based on Bernstein's Theorem, only
decay exponentially. Nevertheless, for $\chi = 20$ the exponential 
bound decays so fast that for large enough column indices (say,
$j\approx 5$ or larger) it is very good for all practical purposes.
\end{example}


\begin{example}\label{example4bis}
Consider the tridiagonal Hermitian Toeplitz matrix $A=(a_{ij}(t))$ of size 
$20\times 20$ over $C([0,1])$ defined by 
\begin{eqnarray*}
&&a_{jj}=1,\qquad j=1,\ldots,20,\\
&&a_{j,j+1}=a_{j+1,j}=1, \qquad j=2k+1,k=1,\ldots,9,\\
&&a_{j,j+1}=a_{j+1,j}=t, \qquad j=4k+2,k=0,\ldots,4,\\
&&a_{j,j+1}=a_{j+1,j}=t^2-1, \qquad j=4k, k=0,\ldots,4.
\end{eqnarray*} 

We scale $A$ so that $\sigma(A)\subset [-1,1]$ and then
compute the Chebyshev approximation $f(A)\approx\sum_{j=0}^{14}c_j A^j$,
where $f(z)=\ln(z+5)$.
The approximation error is bounded in norm by $1.7509\cdot 10^{-14}$.
The decay behavior of $f(A)$, compared with decay 
bounds for different choices of $\chi$ (cf.~Theorem \ref{thm_h2}),
is shown in Figure \ref{fig_ex4bis}.  
The semi-logarithmic
plot clearly shows the exponential decay in the entries
of a row of $f(A)$. 
Note that the decay bounds are somewhat pessimistic in this case.
\end{example}


\begin{figure}
\includegraphics[width=0.52\textwidth]{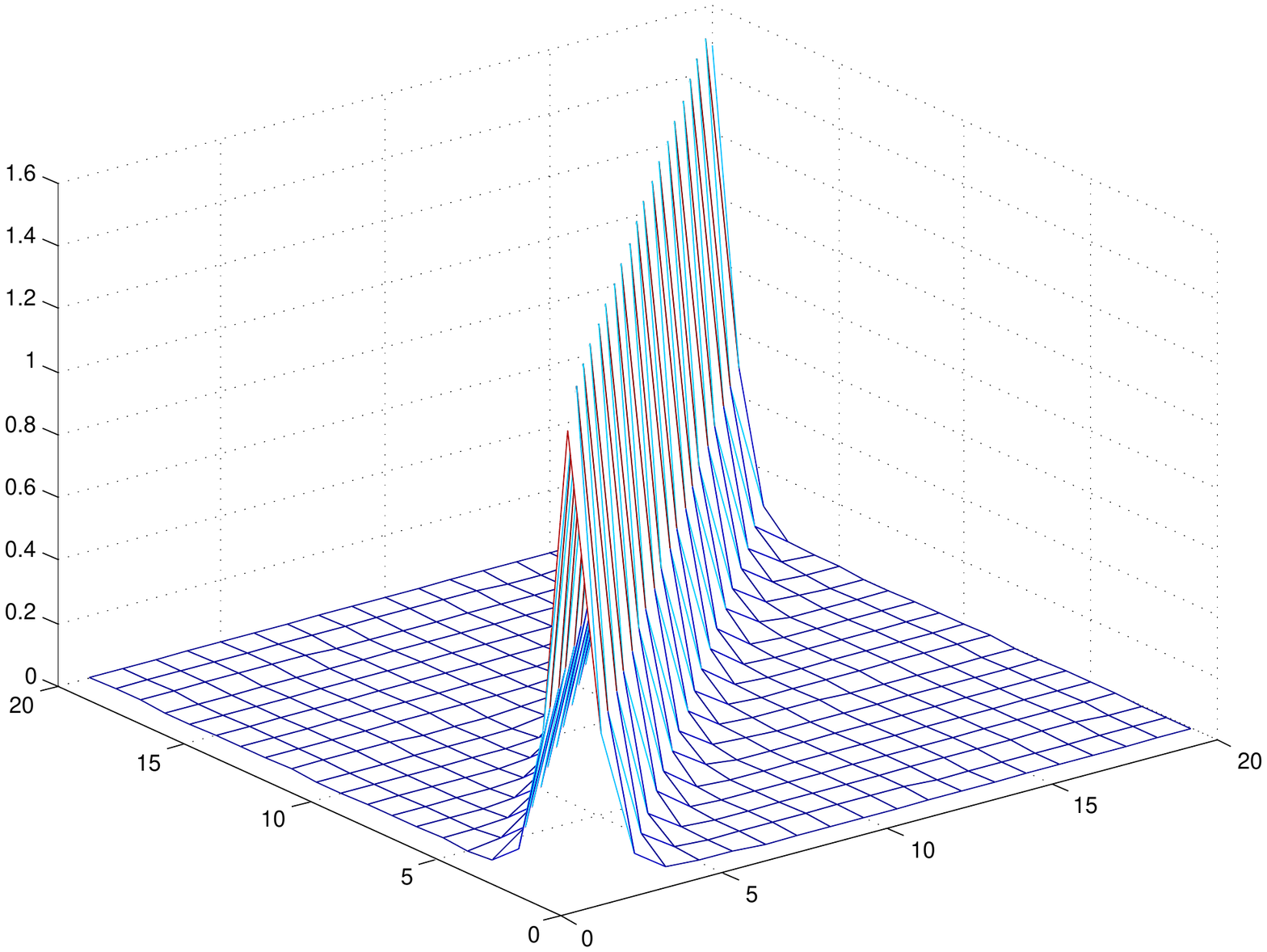}
\includegraphics[width=0.52\textwidth]{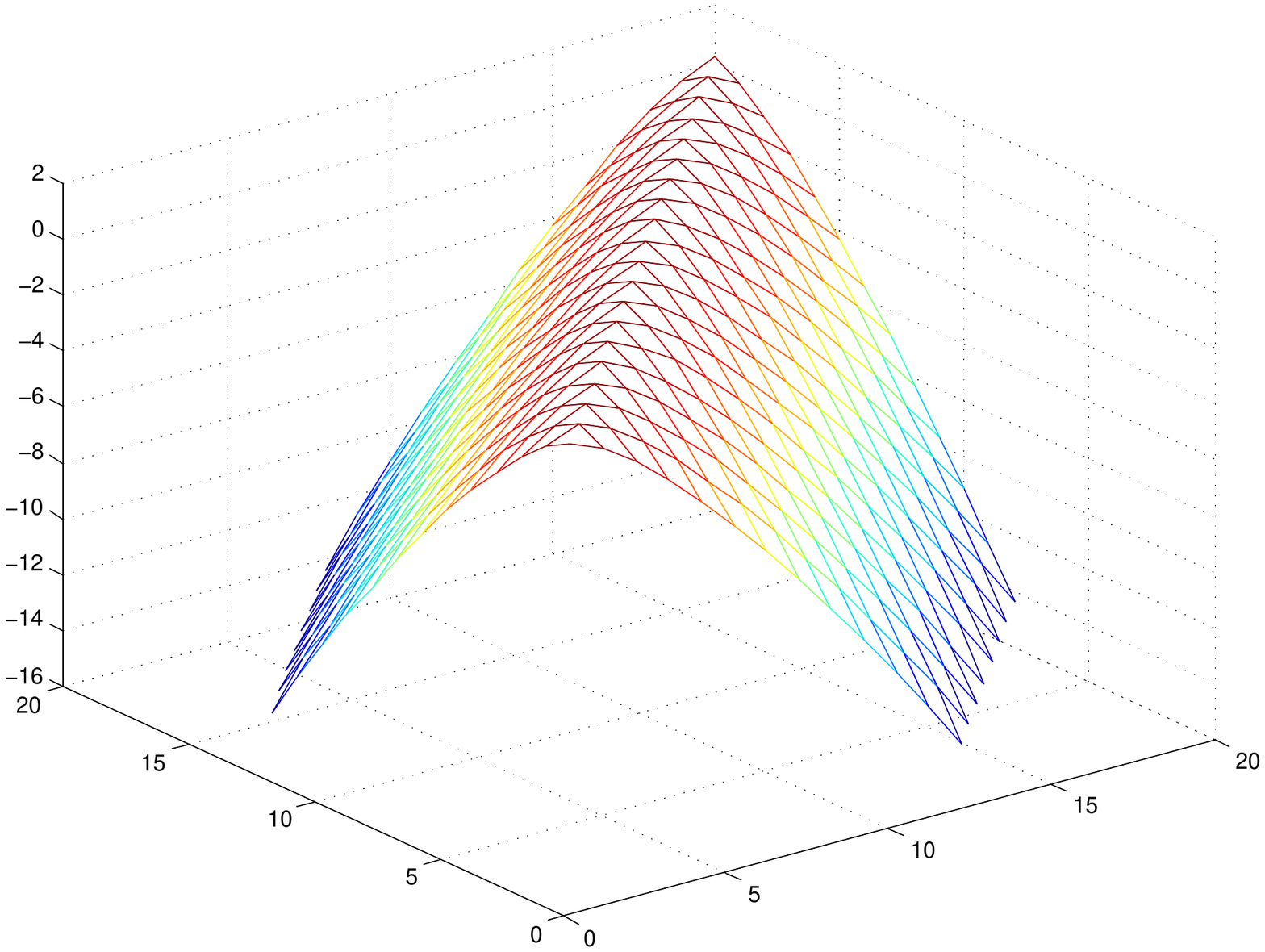}
\caption{Linear mesh plot (left) and $\log_{10}$ mesh plot (right) 
for ${\rm e}^A$ as in Example \ref{example4}.}\label{fig_ex4a}  
\end{figure}

\begin{figure}
\begin{center}
\includegraphics[width=0.8\textwidth]{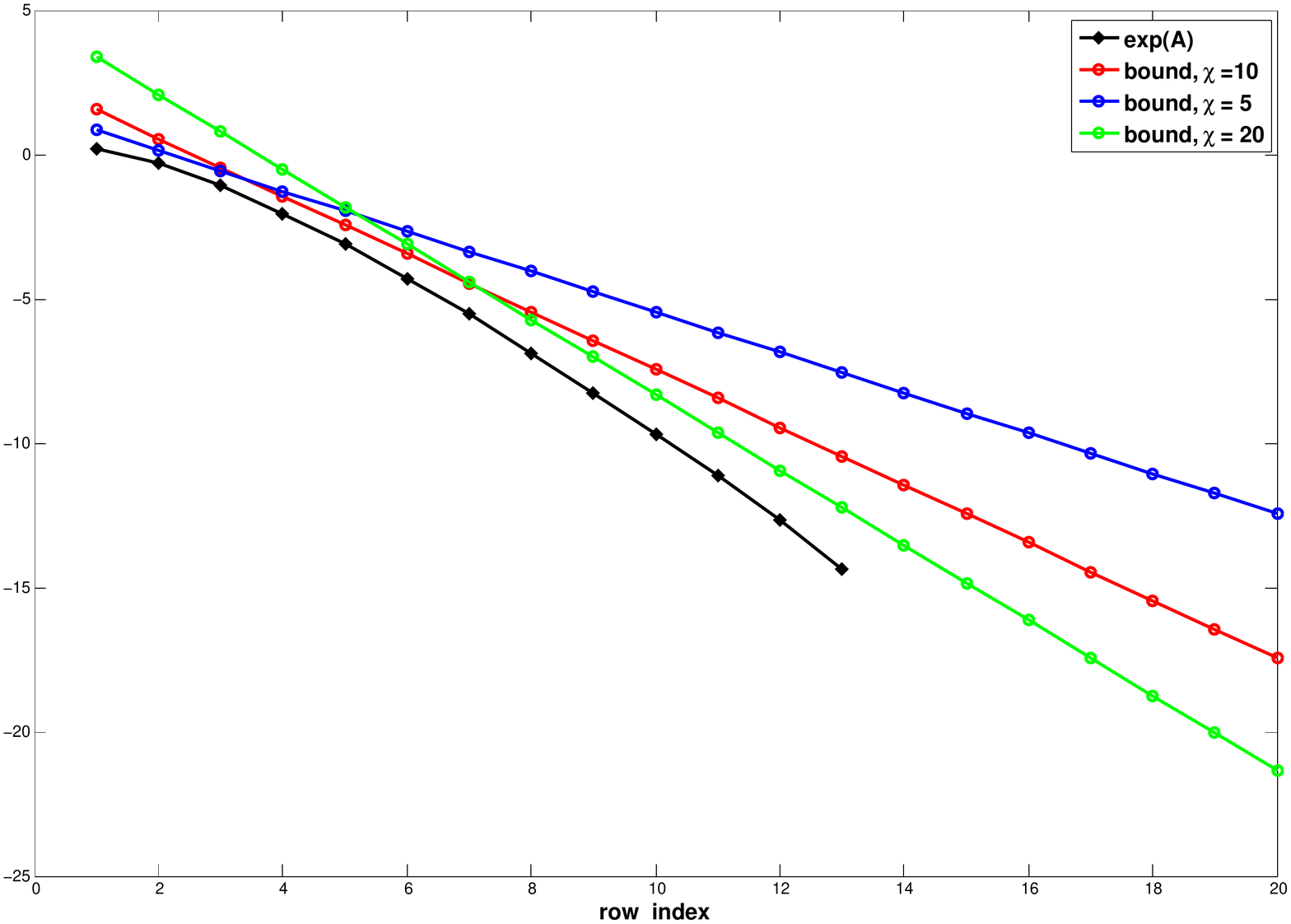}
\caption{Comparison between the first row, in norm, of ${\rm e}^A$ 
as in Example \ref{example4} and theoretical bounds, for 
several values of $\chi$. The vertical axis is shown in 
$\log_{10}$ scale.}\label{fig_ex4b} 
\end{center}
\end{figure}

\begin{figure}
\begin{center}
\includegraphics[width=0.8\textwidth]{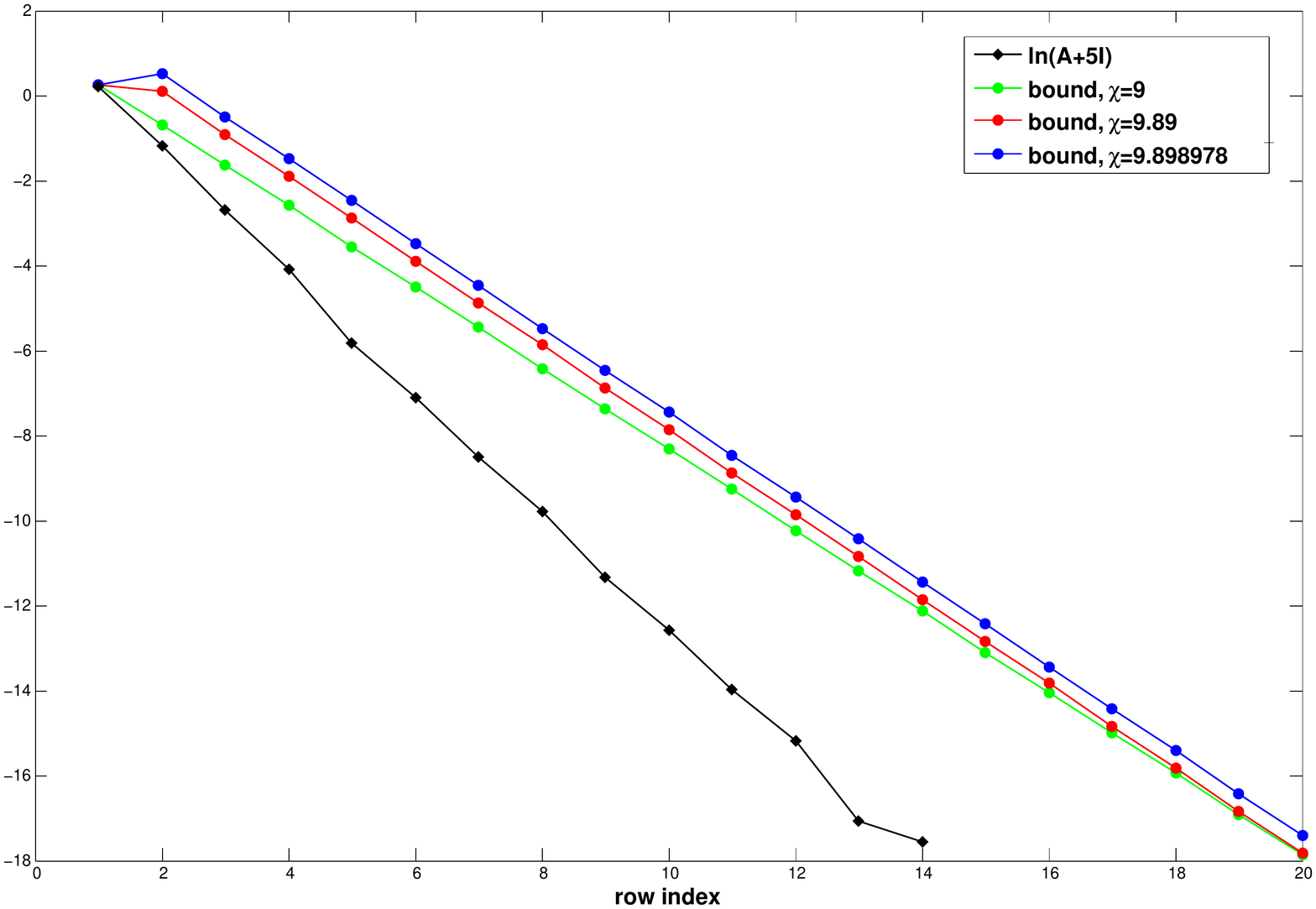}
\caption{Comparison between the 7th row, in norm, of $\ln(A+5I)$ 
as in Example \ref{example4bis} and theoretical bounds, for 
several values of $\chi$. The vertical axis is shown in 
$\log_{10}$ scale.}\label{fig_ex4bis} 
\end{center}
\end{figure}

\subsection{Block matrices}
If we choose $\mathcal{A}_0$ as the noncommutative $C^*$-algebra 
of $k\times k$ complex matrices, then $\mathcal{A}
=\mathbb{C}^{nk\times nk}$ can be 
identified with the $C^*$-algebra of $n\times n$ 
matrices with entries in $\mathcal{A}_0$. 

\begin{example}\label{ex_B}
Let $\mathcal{A}_0=\mathbb{C}^{5\times 5}$ and consider a 
banded non-Hermitian matrix $A$ of size $20\times 20$ with 
entries over $\mathcal{A}_0$. Thus, $A$ is $100\times 100$
as a matrix over $\mathbb{C}$. The entries of each 
block are chosen at random according to a uniform distribution over 
$[-1,1]$. The matrix $A$ has lower bandwidth $2$ and upper 
bandwidth $1$. 
Figure \ref{fig_B} shows the 
sparsity pattern of $A$ and the decay behavior of 
the spectral norms of the blocks of ${\rm e}^{A}$.  
\end{example}


\begin{figure}
\includegraphics[width=0.4\textwidth]{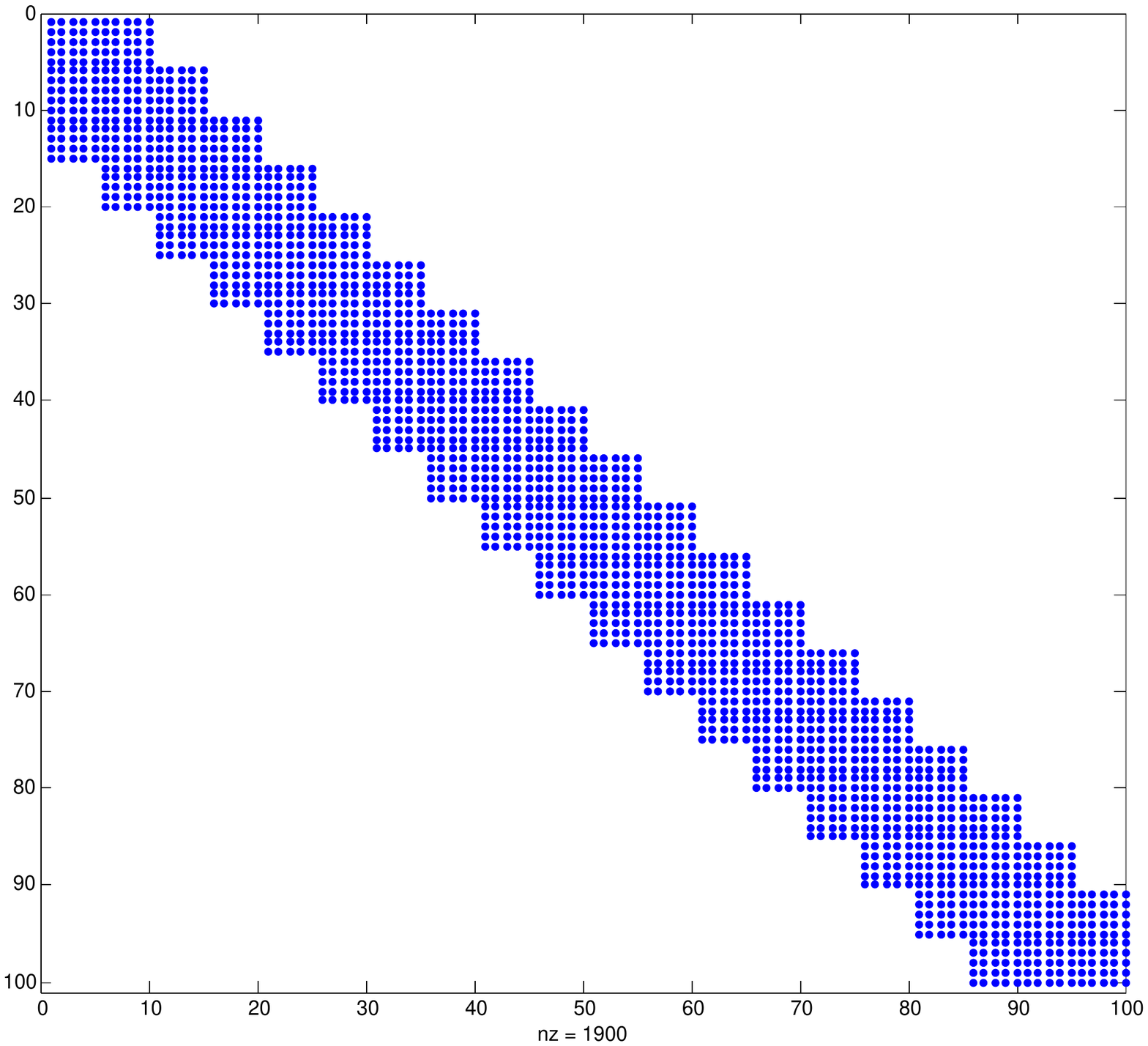}
\includegraphics[width=0.5\textwidth]{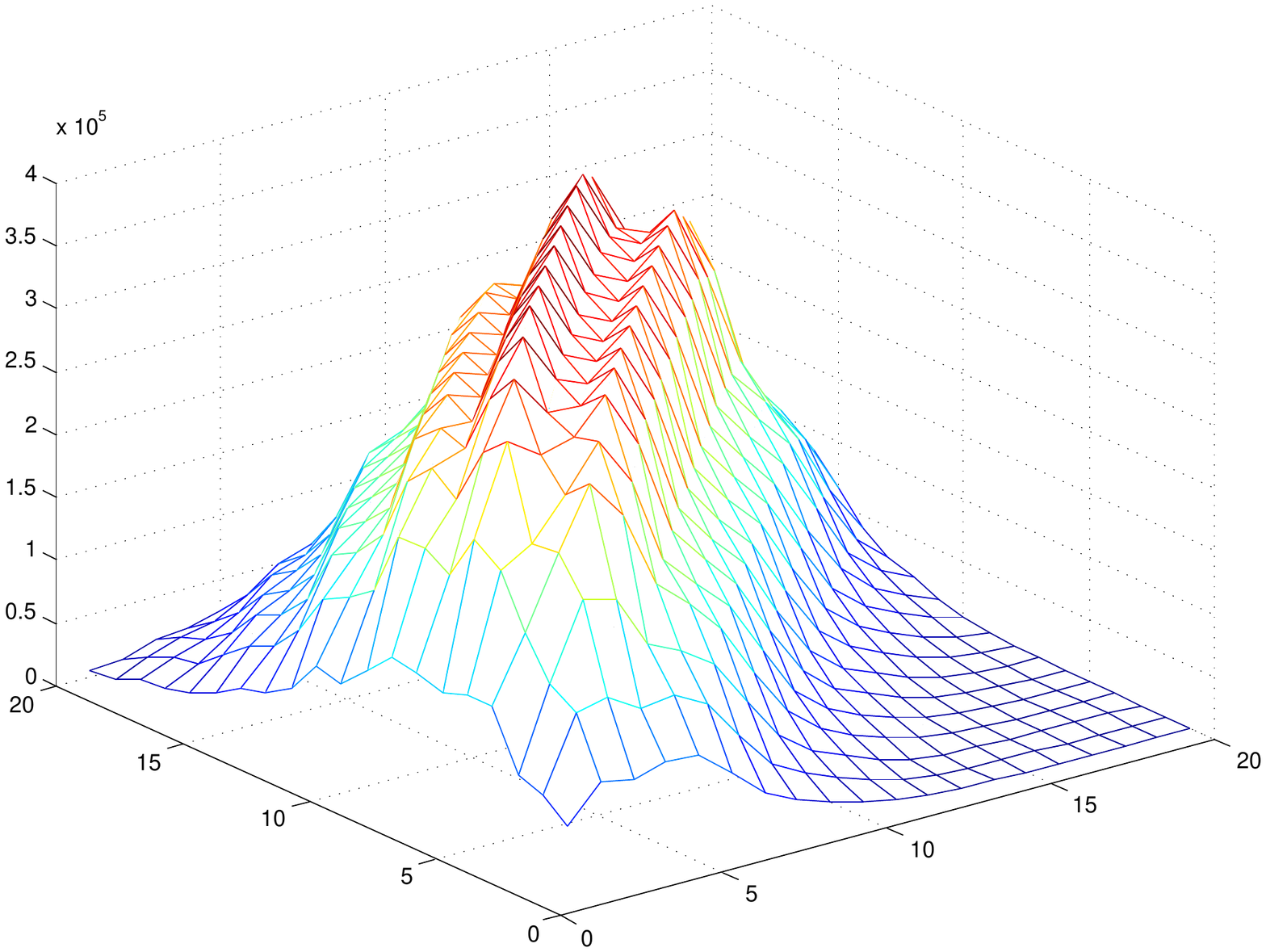}
\caption{Sparsity pattern of $A$ in Example 
\ref{ex_B} (left) and decay behavior in the norms of
the blocks of ${\rm e}^A$ (right).}\label{fig_B}
\end{figure}


\subsection{Matrices over quaternions}
As discussed in section \ref{quater},
we represent matrices over quaternions as complex block 
matrices with blocks of size $2\times 2$.

\begin{example}\label{ex_H1}
 In this example, $A\in\mathbb{H}^{50\times 50}$ is a Hermitian 
Toeplitz tridiagonal matrix with random entries, chosen from a 
uniform distribution over $[-5,5]$. We form the associated 
block matrix $\varphi (A)\in \mathbb{C}^{100\times 100}$, compute 
$f(\varphi (A))$ and convert it back to a matrix in  
$\mathbb{H}^{50\times 50}$. Figure \ref{fig_Hexplog} shows 
the mesh plot of the norms of entries of ${\rm e}^A$ and $\log A $.
\end{example}

\begin{example}\label{ex_H2}
Here we choose $A\in\mathbb{H}^{50\times 50}$ as a Hermitian 
matrix with a more general sparsity pattern and random nonzero 
entries. The decay behavior of ${\rm e}^A$ is shown in Figure 
\ref{fig_Hgeneral}
\end{example}

\begin{figure}
\includegraphics[width=0.55\textwidth]{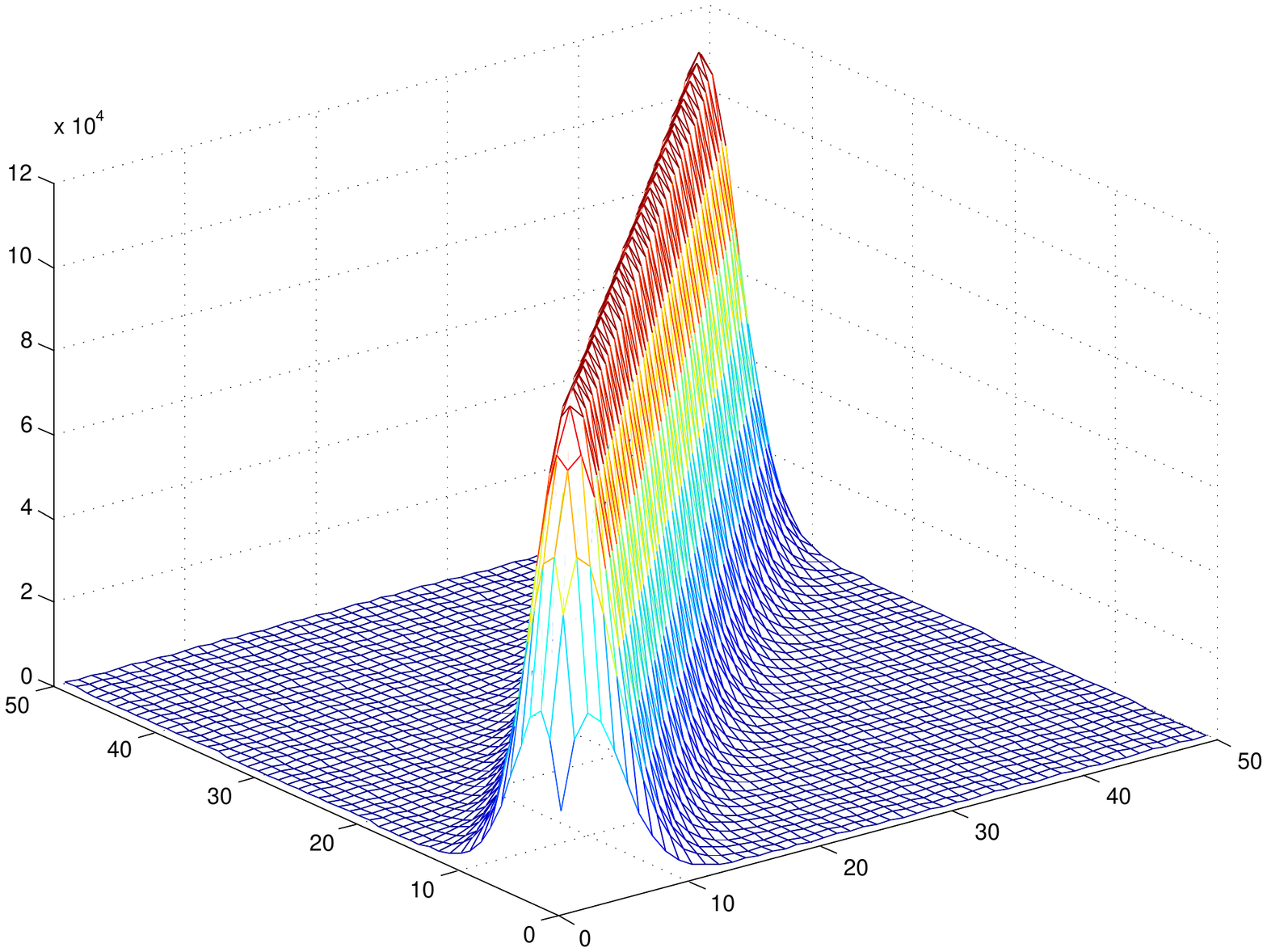}
\includegraphics[width=0.55\textwidth]{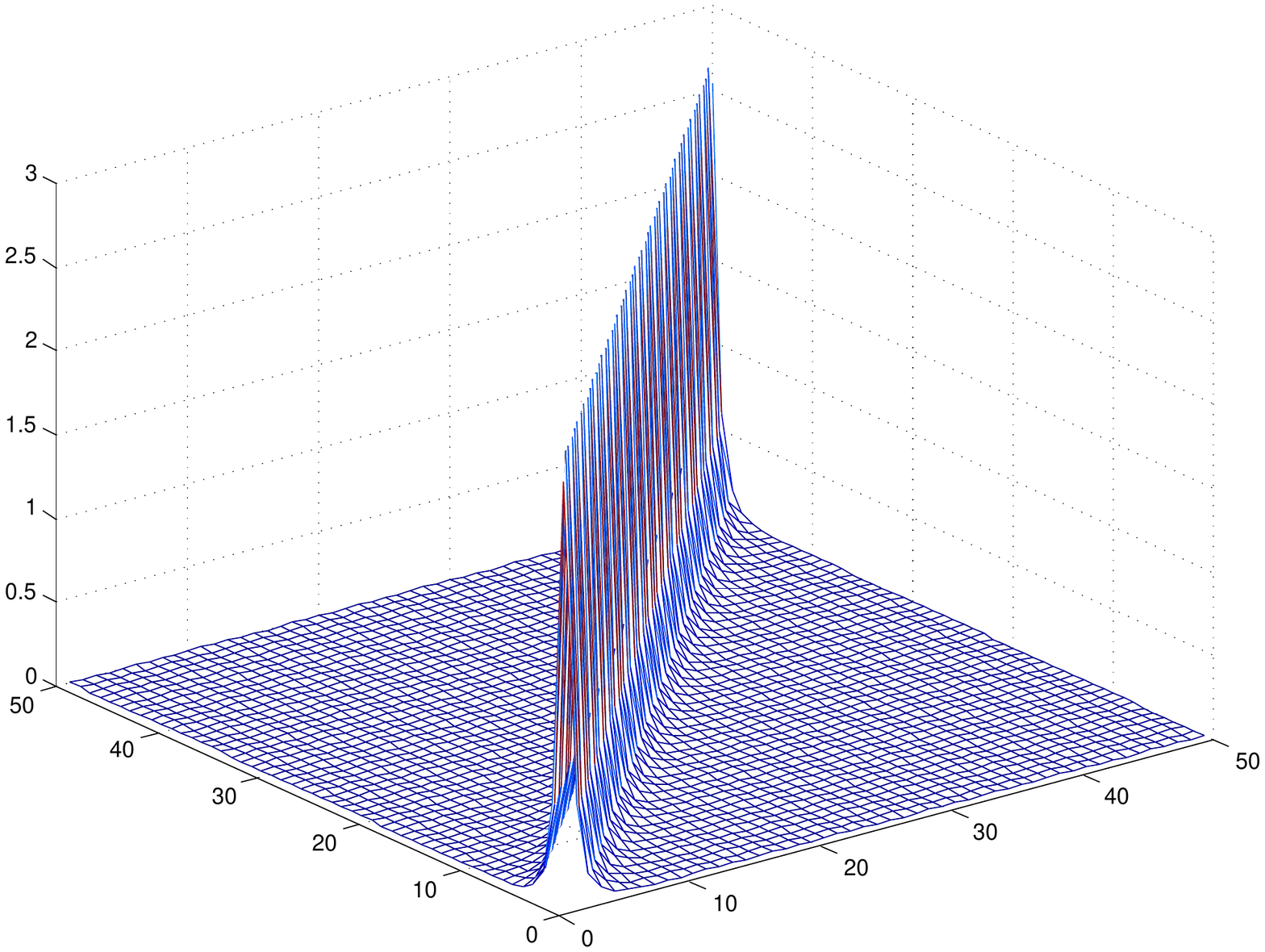}
\caption{Decay behavior for the exponential and the logarithm 
of a tridiagonal matrix over quaternions (Example \ref{ex_H1}).}
\label{fig_Hexplog}
\end{figure}
\begin{figure}
\includegraphics[width=0.45\textwidth]{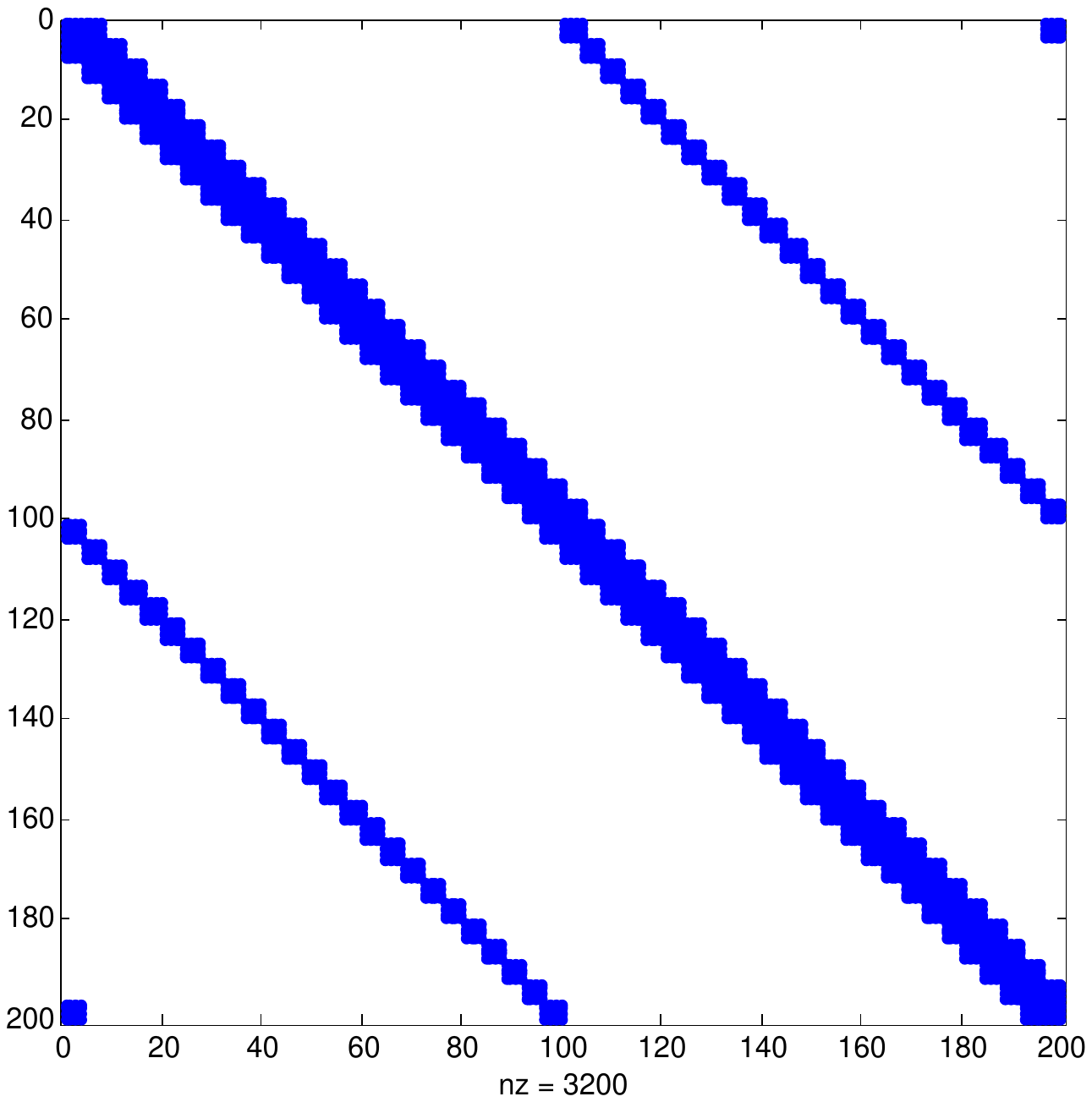}
\includegraphics[width=0.45\textwidth]{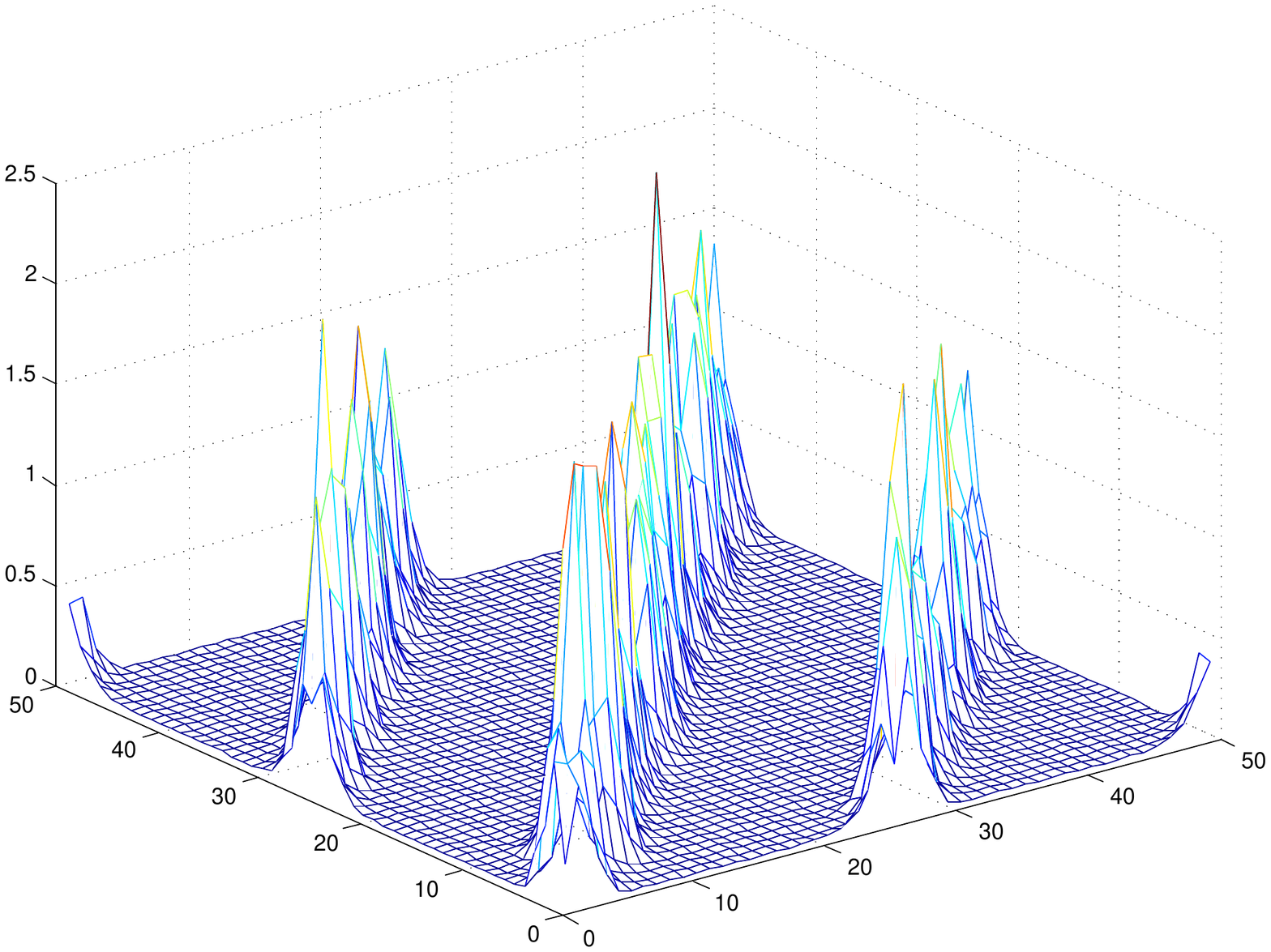}
\caption{Sparsity pattern of $A$ as in Example \ref{ex_H2} (left) 
and decay behavior for ${\rm e}^A$ (right).}\label{fig_Hgeneral}
\end{figure}

\section{Conclusions and outlook}

In this paper we have combined tools from classical approximation theory
with basic results from the general
theory of $C^*$-algebras 
to obtain decay bounds for the entries of
analytic functions of banded or sparse matrices with entries of a
rather general nature. The theory applies to functions of matrices
over the algebra of bounded linear operators on a Hilbert space,
over the algebra of continuous functions on a compact Hausdorff
space, and over the quaternion algebra, thus achieving a
broad generalization of existing exponential decay 
results for functions of matrices over the real or complex fields.
In particular, the theory shows that the exponential decay bounds 
can be extended 
{\em verbatim} to matrices over noncommutative and infinite-dimensional
$C^*$-algebras.

The results in this paper are primarily qualitative in nature,
and the bounds can be pessimistic in practice. This is the price 
one has to pay for the extreme generality of the theory. For entire
functions like the matrix exponential, sharper estimates 
(and superexponential decay bounds)
can be obtained by extending known bounds, such as those
found in \cite{BR09} and in \cite{iserles}.
Another avenue for obtaining more quantitative decay
results is the Banach algebra approach as found, e.g., in
\cite{KSW}.  This approach is quite different from ours.

Future work should address application of the theory
to the derivation of specialized bounds for particular
functions, such as the matrix exponential, and their use 
in problems from physics and numerical analysis.

\vspace{0.1in}

{\bf Acknowledgments.} We would like to thank Pierre--Louis Giscard
(Oxford University) for raising the question that stimulated us to 
write this paper, Bernd Beckermann (University of Lille I)
for relevant bibliographic
indications, and Marc Rybowicz (University of Limoges)
for useful discussions on 
symbolic computation of functions of matrices.

\end{document}